\documentclass[11pt,reqno]{amsart}
\usepackage{amssymb}
\usepackage[all]{xy}
\newtheorem{thm}{Theorem}[section]
\newtheorem{lem}[thm]{Lemma}
\newtheorem{prop}[thm]{Proposition}
\newtheorem{cor}[thm]{Corollary}

\theoremstyle{definition}
\newtheorem{dfn}[thm]{Definition}
\newtheorem{ex}[thm]{Example}

\theoremstyle{remark}
\newtheorem{remark}[thm]{Remark}

\newcommand{\CA}{{\mathcal{A}}}

\newcommand{\CE}{{\mathcal{E}}}

\newcommand{\CH}{{\mathcal{H}}}
\newcommand{\CI}{{\mathcal{I}}}
\newcommand{\CJ}{{\mathcal{J}}}

\newcommand{\CS}{{\mathcal{S}}}

\newcommand{\CL}{{\mathcal{L}}}
\newcommand{\CB}{{\mathcal{B}}}

\newcommand{\CR}{{\mathcal{R}}}

\newcommand{\af}{\alpha}
\newcommand{\bt}{\beta}
\newcommand{\gm}{\gamma}
\newcommand{\dt}{\delta}
\newcommand{\ep}{\varepsilon}

\newcommand{\ld}{\lambda}

\newcommand{\Om}{\Omega}

\newcommand{\N}{{\mathbb{N}}}
\newcommand{\T}{{\mathbb{T}}}

\newcommand{\reg}{{\operatorname{reg}}}

\begin{document}


\title[The simplicity of the  $C^*$-algebras associated to arbitrary labeled spaces]
{The simplicity of the $C^*$-algebras associated to arbitrary labeled spaces}

\author[E. J. Kang]{Eun Ji Kang}
\thanks{This research was supported by Basic Science Research Program through the National Research Foundation of Korea funded by the Ministry of Education
 (NRF-2020R1I1A1A01072970). 
}
\address{Research Institute of Mathematics, Seoul National University, Seoul 08826, 
Korea} \email{kkang3333\-@\-gmail.\-com. }

\subjclass[2000]{37B40, 46L05, 46L55}

\keywords{Labeled graph $C^*$-algebras,  simple $C^*$-algebras, gauge-invariant ideals, generalized Boolean dynamical systems}

\subjclass[2000]{46L05, 46L55}
\begin{abstract}   In this paper, we consider the simplicity of the $C^*$-algebra associated to an arbitrary weakly left-resolving  labeled space $(E,\CL,\CE)$, where  $\CE$ is the smallest non-degenerate accommodating set.  We classify all gauge-invariant ideals of  $C^*(E, \CL,\CE)$ and 
 characterize    minimality of $(E, \CL,\CE)$ in terms of ideal structure of $C^*(E, \CL,\CE)$. Using these results, we prove that 
$C^*(E, \CL,\CE)$ is simple if and only if 
   $(E, \CL,\CE)$ is strongly cofinal and satisfies Condition (L), and 
for any  $ A \in \CE \setminus \{\emptyset\}$ and $B \in \CE$, there is  $C \in \CE_{reg}$ such that  $B \setminus C\in \CH(A)$, and if and only if  $(E, \CL,\CE)$ is minimal and satisfies Condition (L), and if and only if  $(E, \CL,\CE)$ is minimal and satisfies Condition (K).
\end{abstract}

\maketitle

\setcounter{equation}{0}
\section{Introductions}


 A class of $C^*$-algebras $C^*(E)$ associated with directed graphs $E$ was  introduced in \cite{FW, KPRR} as a generalization of the Cuntz-Krieger algebras and there has been  various generalizations of graph $C^*$-algebras.
The $C^*$-algebras associated to ultragraphs \cite{To1}, higher-rank graphs \cite{KP}, subshifts, labeled spaces \cite{BP1}, Boolean dynamical systems \cite{COP} are those generalizations, and generalized Boolean dynamical systems was introduced in \cite{CaK2} to unify $C^*$-algebras of labeled spaces and $C^*$-algebras of Boolean dynamical systems. 
  Among others, we  focus on the $C^*$-algebras associated to arbitrary weakly left-resolving normal labeled spaces. Throughout the paper,  by a labeled space we always mean a weakly left-resolving normal labeled space.
  
  The ideal structure of  $C^*$-algebras of set-finite, receiver set-finite labeled spaces $(E, \CL,\CB)$ with  $E$ having no sinks is now well understood. It is known in \cite[Theorem 5.2]{JKP} that the gauge-invariant ideals of $C^*(E, \CL,\CB)$ are in 
     one-to-one correspondence with the hereditary saturated subsets of $\CB$.   
         We first generalize this results to $C^*$-algebras  of arbitrary labeled spaces. We show that    there is a one-to-one correspondences between
gauge-invariant ideals of $C^*(E, \CL,\CB)$ and
pairs $(\CH,\CS)$ where $\CH$ is a hereditary saturated ideal of
$\CB$ and $\CS$ is an ideal of $\{A\in\CB:
r(A, \af)\in\CH\text{ for all but finitely many }\alpha\}$ such
that $\CH\cup\CB_{reg}\subseteq\CS$ (Theorem \ref{gauge invariant ideal:characterization}). 
 A quotient labeled space was also introduced in \cite[Definition 3.2]{JKP} to 
 realize the quotient algebra $C^*(E, \CL,\CB)/I$ by a gauge-invariant ideal $I$ 
as a $C^*$-algebras of a quotient labeled space. 
    But, a quotient labeled space  is not a labeled space in general, but a Boolean dynamical system. 
        So, in this paper we realize the quotient of $C^*(E, \CL,\CB)$ by $I_{(\CH,\CS)}$ as a $C^*$-algebra of a relative generalized Boolean dynamical system  instead of newly defining $C^*$-algebras of  relative quotient labeled spaces  of arbitrary labeled spaces.
Precisely, we show  that the quotient of  $C^*(E, \CL,\CB)$
by the ideal  $I_{(\CH,\CS)}$ is isomorphic to the
$C^*$-algebra of relative generalized Boolean dynamical system 
 $(\CB/\CH,\CA,\theta,[\CI_{r(\af)}];[\CS])$ (Proposition \ref{isomorphism to quotient}).
These will be easily done by viewing labeled graph $C^*$-algebras as $C^*$-algebras of generalized Boolean dynamical systems and applying results of \cite{CaK2}.

The second goal of the paper is to  investigate the question of when $C^*(E, \CL, \CE)$  is simple, where  $(E, \CL,\CE)$ is an arbitrary weakly left-resolving   labeled space  with  $\CE$ is the smallest non-degenerate accommodating set. 
For an arbitrary graph $E$,  we recall that $E$ is said to satisfy Condition $(L)$ if every loop has as exit, and is said to be cofinal if every vertex connects to every infinite path. 
Then the simplicity of  $C^*(E)$ is characterized as follows.

 \begin{thm}\label{graph case}(\cite[Corollary 2.15]{DT}) Let $E$ be a directed graph.
Then the following are equivalent.
\begin{enumerate}
\item $C^*(E)$ is simple.
\item  The following properties hold:
\begin{enumerate}
\item $E$ is cofinal,
\item $E$ satisfies Condition $(L)$, and 
\item for $w \in E^0$ and $v \in E^0_{sing}$, there is a path $\af \in E^*$ such that $s(\af)=w$ and $r(\af)=v$.
\end{enumerate}
\item  $E$ satisfies Condition (L) and $E^0$ has no proper hereditary  saturated subsets.
\item  $E$ satisfies Condition (K) and $E^0$ has no proper hereditary  saturated subsets.
\end{enumerate}
 \end{thm}
 \noindent
Many authors paid a great deal of attention to extend this result to the $C^*$-algebras associated to set-finite and receiver set-finite labeled spaces $(E,\CL,\CE)$ with $E$ having no sinks or sources. 
 In \cite[Definition 6.1]{BP2} Bates and Pask   introduced a notion of cofinality  appropriate for labeled spaces. 
  In \cite[Definition 3.1]{JK}, a notion of  strong cofinality of  labeled spaces was given to modify minor mistake of results in \cite{BP2}. Then again  a modified version of strong cofinality of a labeled space was introduced in \cite[Definition 2.10]{JP2018}. 
      As an analogue of Condition (L) of usual directed graph, the notion of a disagreeable labeled space was introduced in  \cite[Definition 5.1]{BP2}. On the other hand, the notion of cycle was introduced in \cite[Definition 9.5]{COP} to define 
condition $(L)$ for a labeled space 
(more generally for 
Boolean dynamical systems) which can be regarded as another condition  analogues  
 to Condition  (L) for usual directed graphs. It then is known in  \cite[Proposition 3.7]{JKaP}  that if  $(E,\CL,\CB)$ is disagreeable, then  $(E,\CL,\CB)$  satisfies Condition (L). But, the converse is not true, in general (\cite[Proposition 3.2]{JP2018}). 
   Based on these concepts,  
 it is eventually known  in \cite[Theorem 3.17]{JP2018} that for  a set-finite and receiver set-finite  labeled space $(E,\CL,\CE)$ with $E$ having no sinks or sources,  
$C^*(E,\CL,\CE)$ is simple if and only if  $(E,\CL,\CE)$ is strongly cofinal in the sense of \cite[Definition 2.10]{JP2018} and disagreeable.  
It is also prove in  \cite[Theorem 3.17] {JP2018} that for a labeled space whose Boolean dynamical system satisfies a sort of domain condition, $C^*(E,\CL,\CE)$ is simple if and only if $(E,\CL,\CE)$ satisfies Condition (L) and there are no nonempty hereditary saturated subsets of $\CE$.

 We generalize these results to labeled graph $C^*$-algebras  associated to arbitrary  labeled spaces.
  We first  give an example that shows why we need to change the definition of strong cofinality given in  \cite[Definition 3.1]{JK} to \cite[Definition 2.10]{JP2018}.
We next  prove   that 
  $(E,\CL,\CE)$ is strongly cofinal and for any  $ A \in \CE \setminus \{\emptyset\}$ and $B \in \CE$, there is  $C \in \CE_{reg}$ such that  $B \setminus C\in \CH(A)$  if and only if  
$(E,\CL,\CE)$ is minimal, in the sense that  $\{\emptyset\}$ and $\CE$ are the only hereditary saturated subsets of $\CE$.  
We also identify the minimal condition with the property that  an ideal containing one vertex projection must  be the whole $C^*$-algebra (Theorem \ref{equiv:minimal}).
It will be also shown that if  $(E, \CL,\CE)$ is minimal, then  $(E, \CL,\CE)$ is  disagreeable if and only if 
$(E, \CL,\CE)$   satisfies Condition (L) (Lemma \ref{minimal:prop3}).
As a result, we have our main results as follows. It is 
 a generalization of 
Theorem \ref{graph case}.

 \begin{thm}(Theorem \ref{equiv:simple})  Let $(E,\CL,\CE)$ be a labeled space. 
Then the following are equivalent.
\begin{enumerate}
\item $C^*(E,\CL,\CE)$ is simple.
\item  $(E, \CL,\CE)$ is minimal and  satisfies Condition (L).
\item  $(E, \CL,\CE)$ is minimal and  satisfies Condition (K).
\item The following properties hold:
\begin{enumerate}
\item $(E, \CL,\CE)$ is strongly cofinal,
\item $(E, \CL,\CE)$ satisfies Condition (L), and 
\item for any  $ A \in \CE \setminus \{\emptyset\}$ and $B \in \CE$, there is  $C \in \CE_{reg}$ such that  $B \setminus C\in \CH(A)$. 
\end{enumerate}
\item The following properties hold:
\begin{enumerate}
\item $(E, \CL,\CE)$ is strongly cofinal,
\item $(E, \CL,\CE)$ is disagreeable, and 
\item for any  $ A \in \CE \setminus \{\emptyset\}$ and $B \in \CE$, there is  $C \in \CE_{reg}$ such that  $B \setminus C\in \CH(A)$. 
\end{enumerate}
\end{enumerate}
 \end{thm}

\noindent
    As a corollary, we show  for a set-finite labeled space $(E, \CL,\CE)$ with $E$ having no sinks,  that $C^*(E, \CL,\CE)$ is simple if and only if 
   $(E, \CL,\CE)$ is strongly cofinal and is disagreeable, if and only if $(E, \CL,\CE)$ is strongly cofinal and satisfies Condition (L), if and only if $(E, \CL,\CE)$ is minimal and satisfies Condition (L), and if and only if $(E, \CL,\CE)$ is minimal and satisfies Condition (K).  This generalizes \cite[Theorem 3.7]{JP2018}.

 This paper is organized as follows.  In Section 2 we review basic definitions and terminologies needed for the rest of the paper.   In Section 3 we classify the gauge-invariant ideals in the $C^*$-algebras of arbitrary labeled spaces and describe the quotients as $C^*$-algebras of relative generalized Boolean dynamical systems.
 In Section 4 we examine strong cofinality, minimality and disagreeablity for an arbitrary labeled space, and prove simplicity results for $C^*$-algebras of arbitrary labeled spaces.

\section{Preliminary} 
\subsection{Directed graphs}

A  {\it directed graph} $E=(E^0,E^1,r,s)$
consists of two countable sets of
vertices  $E^0$ and edges $E^1$,
and the range, source maps $r$, $s: E^1\to E^0$. 
A {\it path of length $n$} is a sequence $\ld=\ld_1 \cdots \ld_n$ of edges
  such that $r(\ld_{i})=s(\ld_{i+1})$ for $ 1\leq i\leq n-1$.
We write $|\ld| = n$ for the length of $\ld $ and the vertices in $E^0$ are regarded as finite paths of length zero. By $E^n$  we mean the set of all  paths of length $n$. 
The maps $r,s$ naturally extend to the set 
$E^{\geq 0}:=\cup_{n \geq 0} E^n$ of all finite paths, 
where $r(v)=s(v)=v$ for $v \in E^0$. 
We denote by  $E^{\infty}$ the set of all infinite paths 
$x=\ld_{1}\ld_{2}\cdots$, $\ld_{i}\in E^1$ with 
$r(\ld_i)=s(\ld_{i+1})$ for $i\geq 1$, 
and define $s(x):= s(\ld_1)$. We als use notation like 
$E^{\leq n}$  and $E^{\geq n}$ which should have their obvious meaning.

A vertex $v\in E^0$ 
 is called a {\it source} if $|r^{-1}(v)|=0$ and $v$
  is called a {\it sink } if $|s^{-1}(v)|=0$, and $v$ is called an {\it infinite-emitter} if $|s^{-1}(v)|=\infty$. We define $E^0_{sink}$ to be the set of all sinks in $E^0$. We let $E^0_{reg}=\{v \in E^0:0 < |s^{-1}(v)| < \infty \}$ and let $E^0_{sing}=E^0 \setminus E^0_{reg}$. 
 
 A finite path $\ld=\ld_1 \cdots \ld_{|\ld|}\in E^{\geq 1}$ with $r(\ld)=s(\ld)$ is 
called a {\it loop}, and  
an {\it exit} of a loop $\ld$ is a path 
$\dt\in E^{\geq 1}$ such that  
$|\dt|\leq |\ld|,\ s(\dt)=s(\ld), \text{ and } \dt\neq \ld_1 \cdots \ld_{|\dt|}$. 
A graph $E$ is said to satisfy {\it Condition} (L)  
if every loop has an exit and $E$ is said to satisfy {\it Condition} (K) 
if every vertex $v \in E^0$ lies on no loops, or if there are two loops $\af$ and $\bt$ such that $s(\af)=s(\bt)=v$ and neither $\af$ nor $\bt$ is an initial path of the other.

\subsection{Labeled spaces}

A {\it labeled graph} $(E,\CL)$ over a countable alphabet  $\CA$
 consists of a directed graph $E$ and
 a {\it labeling map} $\CL:E^1\to \CA$.
 We assume that the map $\CL$ is onto. 
By $\CA^*$ and $\CA^\infty$, 
 we denote respectively the sets of all finite words
  and infinite words in symbols of $\CA$.   
 To each finite path $\ld=\ld_1\cdots \ld_n\in E^n$, there corresponds 
 a finite labeled path 
 $\CL(\ld):=\CL(\ld_1)\cdots \CL(\ld_n)\in \CL(E^n)\subset \CA^*$, and  
 similarly an infinite labeled path 
 $\CL(x):=\CL(\ld_1)\CL(\ld_2)\cdots\in \CL(E^\infty)\subset\CA^\infty$  
 to each infinite path $x=\ld_1\ld_2\cdots \in E^\infty$.  
We use notation $\CL^*(E):=\CL(E^{\geq 1})$, where $E^{\geq 1}=\cup_{n \geq 1} E^n$. We denote the subpath $\af_i\cdots \af_j$ of 
$\af=\af_1\af_2\cdots\af_{|\af|}\in \CL(E^{\geq 1})$
by $\af_{[i,j]}$ for $1\leq i\leq j\leq |\af|$.
A subpath of the form $\af_{[1,j]}$ is called
an {\it initial path} of $\af$. 
The {\it range} $r(\af)$ 
 of a labeled path $\af\in \CL^*(E)$ is a 
 subset of $E^0$ defined by
  $$ r(\af) =\{r(\ld) \,:\, \ld\in E^{\geq 1},\,\CL(\ld)=\af\}.$$
 The {\it relative range of $\af\in \CL^*(E)$
 with respect to $A \subset E^0$} is defined to be
$$
 r(A,\af)=\{r(\ld)\,:\, \ld\in E^{\geq 1},\ \CL(\ld)=\af,\ s(\ld)\in A\}.
$$

Let   $\CB\subseteq 2^{E^0}$ be a collection of subsets of $E^0$. We say $\CB$
is 
 {\it closed under relative ranges} for $(E,\CL)$ if $r(A,\af)\in \CB$ for all $A\in \CB$ and $\af\in  \CL^*(E)$.
 We call $\CB$ an {\it accommodating set} for $(E,\CL)$
 if it satisfies
 \begin{enumerate}
 \item[(i)] $r(\af) \in \CB$ for all $\af\in  \CL^*(E)$,
 \item[(ii)] it is closed under relative ranges,
 \item[(iii)] it is closed under finite intersections and unions.
  \end{enumerate} 
  \noindent If, in addition, $\CB$  is closed under relative complements, then  $\CB$ is said to be {\it non-degenerate}.    
  The triple $(E,\CL,\CB)$ is called  a {\it labeled space} when 
$\CB$ is accommodating for $(E,\CL)$.
  Moreover, if  $\CB$ is non-degenerate, then $(E,\CL, \CB)$ is called  {\it normal} as in \cite{BCP}. By $\CE$, we denote the smallest non-degenerate accommodating set 
for a labeled graph $(E,\CL)$. 

A labeled space $(E,\CL,\CB)$ is said to be {\it weakly left-resolving} if it satisfies
 $$r(A,\af)\cap r(B,\af)=r(A\cap B,\af)$$
  for all $A,B\in \CB$ and  $\af\in \CL^*(E)$.
  A labeled graph $(E,\CL)$ is {\it left-resolving} if 
  $\CL : r^{-1}(v) \rightarrow \mathbf{\CA}$ is injective 
 for each $v \in E^0$.
Left-resolving labeled spaces are weakly left-resolving.

\vskip 1pc
\noindent
{\bf Assumptions.}\label{assumptions-sink} In this paper, $(E,\CL,\CB)$  is always weakly left-resolving normal and $\CL: E^1 \to \CA$ is  onto.

\vskip 1pc

 For $A,B \subset E^0$ and $n\geq 1$, let
 $$ AE^n =\{\ld\in E^n\,:\, s(\ld)\in A\},\ \
  E^nB=\{\ld\in E^n\,:\, r(\ld)\in B\}.$$
  A labeled space $(E,\CL,\CB)$ is said to be {\it set-finite}
 ({\it receiver set-finite}, respectively) if for every $A\in \CB$ and $n \geq 1$ 
 the set  $\CL(AE^n)$ ($\CL(E^nA)$, respectively) is finite.   
   We also say that $A \in \CB$ is {\it regular} if $0 < |\CL(BE^1)| < \infty$ for any $\emptyset \neq B \in \CB$ with $B \subseteq A$. If $A \in \CB$  is not regular, then it is called a {\it singular} set. We write $\CB_{reg}$ for the set of all regular sets.  
Note that if $E$ has no sinks, then $(E, \CL,\CB)$ is set-finite if and only if $\CB=\CB_{reg}$. 
A set $A \in \CB$ is called {\it minimal} (in $\CB$) if $A \cap B $ is either $A$
  or $\emptyset$ for all $B \in \CB$. 
  
Let $(E,\CL)$ be a labeled graph and 
let $\Om_0(E)$ be the set of all vertices that are not sources, and let $l\geq 1$.
The relation $\sim_l$ on  $\Om_0(E)$, given by 
$v\sim_l w$ if and only of $\CL(E^{\leq l} v)=\CL(E^{\leq l} w)$, is  
an equivalence relation, and the equivalence class 
$[v]_l$ of $v\in \Om_0(E)$ is called a {\it generalized vertex}. 
  If $k>l$,  then $[v]_k\subset [v]_l$ is obvious and
   $[v]_l=\cup_{i=1}^m [v_i]_{l+1}$
   for some vertices  $v_1, \dots, v_m\in [v]_l$ (\cite[Proposition 2.4]{BP2}). 
The generalized vertices of labeled graphs play the role of vertices in usual graphs.  


\subsection{$C^*$-algebras of labeled spaces} We review the definition of the $C^*$-algebras associated to labeled spaces 
from \cite{BCP}.

\begin{dfn}\label{rep}(\cite[Definition 2.1]{BCP}) Let $(E, \CL, \CB)$ be a labeled space. 
A {\it representation} of $(E,\CL,\CB)$
is a family of projections $\{p_A\,:\, A\in \CB\}$ and
partial isometries 
$\{s_\af \,:\, \af\in \CA\}$ such that for $A, B\in \CB$ and $\af,\bt\in \CA$,
\begin{enumerate}
\item[(i)]  $p_{\emptyset}=0$, $p_{A\cap B}=p_A p_B$, and
$p_{A\cup B}=p_A+p_B-p_{A\cap B}$,
\item[(ii)] $p_A s_\af=s_\af p_{r(A,\af)}$,
\item[(iii)] $s_\af^*s_\af=p_{r(\af)}$ and $s_\af^* s_\bt=0$ unless $\af=\bt$,
\item[(iv)]\label{CK4}  $p_A=\sum_{\af \in \CL(AE^1)} s_\af p_{r(A,\af)}s_\af^*$ for $A \in \CB_{reg}$. 
\end{enumerate}
\end{dfn}

\begin{remark} 
 Let $(E, \CL, \CB)$ be a  labeled space. 
  \begin{enumerate}
 \item For any $A \in \CB$, the condition 
$$ |\CL(AE^1)| < \infty ~\text{and}~A \cap B = \emptyset ~\text{for all}~ B \in \CB ~\text{satisfying}~  B \subseteq E^0_{sink}$$
is equivalent to $A \in \CB_{reg}$. Thus, the condition (iv) in Definition \ref{rep} is equivalent to (iv) in \cite[Definition 2.1]{BCP}.
\item If $E$ has no sinks,   the condition (iv) in Definition \ref{rep} is equivalent to
$$p_A=\sum_{\af \in \CL(AE^1)} s_\af p_{r(A,\af)}s_\af^* ~\text{for }~ A \in \CB ~\text{with}~|\CL(AE^1)| < \infty.$$
\end{enumerate}
\end{remark}

Given a  labeled space  $(E,\CL,\CB)$,  it is known in \cite[Theorem 3.8]{BCP} that 
there exists a $C^*$-algebra $C^*(E,\CL,\CB)$ generated by 
a universal representation $\{s_\af,p_A\}$ of $(E,\CL,\CB)$. 
We call $C^*(E,\CL,\CB)$ the {\it labeled graph $C^*$-algebra} of
a labeled space $(E,\CL,\CB)$ 
and simply write $C^*(E,\CL,\CB)=C^*(s_\af,p_A)$ 
to indicate the generators. 
Note that   $s_\af\neq 0$ and $p_A\neq 0$ for $\af \in \CA$
and  $A\in \CB$, $A\neq \emptyset$.

\vskip 1pc

\begin{remark} \label{rmk1}
Let $(E,\CL,\CB)$ be a labeled space. For notational convenience, we use a symbol  $\epsilon$ such that 
$r(\epsilon) =E^0$, $r(A, \epsilon) = A$ for all $A \subset E^0$,
and $\af=\epsilon\af=\af\epsilon$ for all $\af\in \CL(E^{\geq 1})$, and write 
$\CL^\#(E):=\CL(E^{\geq 1})\cup \{\epsilon \}.$ 
For  $\ep \in \CL^{\#}(E)$, let $s_\ep$ denote the unit of the multiplier algebra of $C^*(E,\CL,\CB)$. 

\begin{enumerate}
\item 
We have the following equality  
$$(s_\af p_{A} s_\bt^*)(s_\gm p_{B} s_\dt^*)=\left\{
                      \begin{array}{ll}
                        s_{\af\gm'}p_{r(A,\gm')\cap B} s_\dt^*, & \hbox{if\ } \gm=\bt\gm' \\
                        s_{\af}p_{A\cap r(B,\bt')} s_{\dt\bt'}^*, & \hbox{if\ } \bt=\gm\bt'\\
                        s_\af p_{A\cap B}s_\dt^*, & \hbox{if\ } \bt=\gm\\
                        0, & \hbox{otherwise,}
                      \end{array}
                    \right.
$$                    
for $\af,\bt,\gm,\dt\in  \CL^{\#}(E)$ and $A,B\in \CB$ (see \cite[Lemma 4.4]{BP1}). 
Since 
$s_\af p_A s_\bt^*\neq 0$ if and only if $A\cap r(\af)\cap r(\bt)\neq \emptyset$, 
it follows that  
\begin{eqnarray}\label{eqn-elements}
\hskip 3pc 
C^*(E,\CL,\CB)=\overline{\rm span}\{s_\af p_A s_\bt^*\,:\,
\af,\,\bt\in  \CL^{\#}(E) ~\text{and}~ A \subseteq r(\af)\cap r(\bt)\}. 
\end{eqnarray}  

\item By universal property of  $C^*(E,\CL,\CB)=C^*(s_\af, p_A)$, there is a strongly continuous action $\gm:\mathbb{T} \rightarrow Aut(C^*(E,\CL,\CB))$,  
called the  {\it gauge action}, such that
$$\gm_z(s_\af)=zs_\af ~ \text{ and } ~ \gm_z(p_A)=p_A$$
for $\af\in \CA$ and $A\in \CB$. 

\end{enumerate}
\end{remark}

The notion of cycle was introduced (\cite[Definition 9.5]{COP}) to define 
condition $(L)$ for a labeled space $(E,\CL,\CB)$ 
(more generally for 
Boolean dynamical systems) which is an analogue   
 to Condition  (L) for usual directed graphs. 
 
\begin{dfn} (\cite[Definition 9.5]{COP}) Let $(E,\CL,\CB)$ be a labeled space. 
\begin{enumerate}
\item A pair  $(\af,A)$ with $\af \in \CL^*(E)$ and $ \emptyset \neq A \in \CB$ is a  {\it cycle} if 
$B=r(B,\af)$ for every  $B\in \CB$ with  $B \subseteq A$.
\item A cycle $(\af,A)$ has an \emph{exit} if for there is a $t \leq |\af|$ and a $B\in\CB$ such that $\emptyset \neq B \subseteq r(A, \af_{[1,t]})$ and $\CL(BE^1)\ne\{\af_{t+1}\}$ (where $\af_{|\af|+1}:=\af_1$).
\item A cycle $(\af,A)$ has no exits if for all $t \leq |\af|$ and all $\emptyset \neq B \subseteq r(A, \af_{[1,t]})$, we have $B \in \CB_{reg}$ with $\CL(BE^1)=\{\af_{t+1}\}$ (where $\af_{|\af|+1}:=\af_1$).
\item A labeled space $(E,\CL,\CB)$ satisfies  
{\it Condition $(L)$} if every cycle has an exit.
\end{enumerate}
\end{dfn}

\begin{thm}\label{CKUT}{\rm (The Cuntz-Krieger Uniqueness theorem \cite[Theorem 5.5]{BP1}, \cite[Theorem 9.9]{COP})} Let  $\{t_a, q_A\}$ be a representation of a labeled space $(E,\CL,\CB)$ 
such that $q_A\neq 0$ for all nonempty $A\in \CB$. 
If $(E,\CL,\CB)$ satisfies condition  $(L)$, then 
the canonical homomorphism $\phi:C^*(E,\CL,\CB)=C^*(s_a, p_A)\to 
 C^*(t_a, q_A)$ such that $\phi(s_a)=t_a$ and 
$\phi(p_A)=q_A$  is an isomorphism. 
\end{thm}

\subsection{Generalized Boolean dynamical systems and their $C^*$-algebras}\label{GBDS}
For details of the following, we refer the reader to \cite{ CaK1, CaK2, COP}.

Let $\CB$ be a Boolean algebra. 
A non-empty subset $\CI$ of $\CB$ is called  an {\em ideal} \cite[Definition 2.4]{COP} if 
\begin{enumerate}
\item[(i)] if $A, B \in \CI$, then $A \cup B \in \CI$,
\item[(ii)] if $A \in \CI$ and $ B \in \CB$, then   $A \cap B \in \CI$. 
\end{enumerate}
An ideal $\CI$ of a Boolean algebra $\CB$ is a Boolean subalgebra. For $A \in \CB$, the ideal generated by $A$ is defined by $\CI_A:=\{ B \in \CB : B \subseteq A\}.$

A {\em Boolean dynamical system} is a triple $(\CB,\CL,\theta)$ where $\CB$ is a Boolean algebra, $\CL$ is a set, and $\{\theta_\af\}_{\af \in \CL}$ is a set of actions on $\CB$ such that for $\af=\af_1 \cdots \af_n \in \CL^*\setminus\{\emptyset\}$, the action $\theta_\af: \CB \rightarrow \CB$ is defined as $\theta_\af:=\theta_{\af_n} \circ \cdots \circ \theta_{\af_1}$.  We  also define $\theta_\emptyset:=\text{Id}$.

For $B \in \CB$, we define
\[
\Delta_B^{(\CB,\CL,\theta)}:=\{\af \in \CL:\theta_\af(B) \neq
\emptyset \} ~\text{and}~  \ld_B^{(\CB,\CL,\theta)}:=|\Delta_B^{(\CB,\CL,\theta)}|.
\]
We will often just write $\Delta_B$ and $\ld_B$ instead of
$\Delta_B^{(\CB,\CL,\theta)}$ and $\ld_B^{(\CB,\CL,\theta)}$.
We say that $A \in \CB$ is {\em regular} (\cite[Definition 3.5]{COP})
if for any $\emptyset \neq B \in \CI_A$, we have $0 < \ld_B < \infty$.
If $A \in \CB$ is not regular, then it is called a {\em singular} set.
We write $\CB^{(\CB,\CL,\theta)}_{reg}$ or just $\CB_{reg}$ for the
set of all regular sets. Notice that $\emptyset\in\CB_{reg}$.

Let 
$
\mathcal{R}_\alpha:=\mathcal{R}_\alpha^{(\CB,\CL,\theta)}=\{A\in\mathcal{B}:A\subseteq\theta_\alpha(B)\text{ for some }B\in\mathcal{B}\}
$
for each $\alpha \in \mathcal{L}$. 
A {\em generalized Boolean dynamical system} is a quadruple  $(\CB,\CL,\theta,\CI_\alpha)$ where  $(\CB,\CL,\theta)$ is  a Boolean dynamical system  and  $\{\CI_\alpha:\alpha\in\CL\}$ is a family of ideals in $\CB$ such that $\CR_\alpha\subseteq\CI_\alpha$ for each $\alpha\in\CL$. A {\em  relative generalized Boolean dynamical system} is a pentamerous $(\CB,\CL,\theta,\CI_\alpha;\CJ)$ where  $(\CB,\CL,\theta,\CI_\alpha)$ is a  generalized Boolean dynamical system  and  $\CJ$ is an ideal of  $\CB_{reg}$ (\cite[Definition 3.2]{CaK2}).

\begin{dfn}(\cite[Definition 3.2]{CaK2})\label{def:representation of RGBDS} 
Let $(\CB,\CL,\theta, \CI_\af; \CJ)$ be a relative generalized Boolean dynamical system. A {\it  $(\CB, \CL, \theta, \CI_\af;\CJ)$-representation} is a family of projections $\{P_A:A\in\mathcal{B}\}$ and a family of partial isometries $\{S_{\alpha,B}:\alpha\in\mathcal{L},\ B\in\mathcal{I}_\alpha\}$ such that for $A,A'\in\mathcal{B}$, $\alpha,\alpha'\in\mathcal{L}$, $B\in\mathcal{I}_\alpha$ and $B'\in\mathcal{I}_{\alpha'}$,
\begin{enumerate}
\item[(i)] $P_\emptyset=0$, $P_{A\cap A'}=P_AP_{A'}$, and $P_{A\cup A'}=P_A+P_{A'}-P_{A\cap A'}$;
\item[(ii)] $P_AS_{\alpha,B}=S_{\alpha,  B}P_{\theta_\af(A)}$;
\item[(iii)] $S_{\alpha,B}^*S_{\alpha',B'}=\delta_{\alpha,\alpha'}P_{B\cap B'}$;
\item[(iv)] $P_A=\sum_{\af \in\Delta_A}S_{\af,\theta_\af(A)}S_{\af,\theta_\af(A)}^*$ for all  $A\in \mathcal{J}$.
\end{enumerate}
\end{dfn}

Given a $(\CB, \CL, \theta, \CI_\af;\CJ)$-representation $\{P_A, S_{\af,B}\}$ in a $C^*$-algebra $\CA$, we denote by $C^*(P_A, S_{\af,B})$ the $C^*$-subalgebra of $\CA$ generated by $\{ P_A,  S_{\af,B}\}$.
It is shown in \cite{CaK2} that there exists  a universal $(\CB, \CL, \theta, \CI_\af;\mathcal{J})$-representation $\{p_A, s_{\af,B}: A\in \CB, \af \in \CL ~\text{and}~ B \in \CI_\af\}$   in a  $C^*$-algebra. 
 We write $C^*(\mathcal{B},\mathcal{L},\theta, \CI_\af;\mathcal{J})$ for $C^*(p_A,s_{\af,B})$ and    call it the {\it  $C^*$-algebra of $(\CB,\CL,\theta,\CI_\alpha;\CJ)$}.

By a {\it Cuntz--Krieger representation of $(\CB, \CL,
\theta,\CI_\af)$} we mean a $(\CB, \CL, \theta,\CI_\af;
\CB_{reg})$-representation. We write $C^*(\CB,\CL,\theta, \CI_\af)$
for $C^*(\CB, \CL, \theta, \CI_\af;\CB_{reg})$ and call it the {\it
$C^*$-algebra of $(\CB,\CL,\theta,\CI_\alpha)$}.
 When $(\CB,\CL,\theta)$ is a Boolean dynamical system, then we write
$C^*(\CB,\CL,\theta)$ for $C^*(\CB,\CL,\theta, \CR_\af)$ and call it
the \emph{$C^*$-algebra of $(\CB,\CL,\theta)$}. 

\subsubsection{Viewing labeled graph $C^*$-algebras as $C^*$-algebras of generalized Boolean dynamical systems}
We view labeled graph $C^*$-algebras as $C^*$-algebras of generalized Boolean dynamical systems. 
Let $(E,\CL,\CB)$ be a  labeled space where $\CL:E^1 \to \CA$ is onto and put $C^*(E, \CL, \CB)=C^*(p_A, s_\af)$. Then $\CB$ is a Boolean algebra  and for each $\af \in \CA$, the map $\theta_{\af}:\CB \to \CB$ defined by 
$\theta_\af(A):=r(A,\af)$
is an action on $\CB$ (\cite[Example 11.1]{COP}).  Put $\CR_\af=\{A \in \CB: A \subseteq r(B,\af) ~\text{for some}~ B \in \CB\}$ and let $$\CI_{r(\af)}=\{A \in \CB : A \subseteq r(\af)\}.$$
It is clear that $\CR_\af \subseteq \CI_{r(\af)}$  for each $\af \in \CL$. Then $(\CB, \CA,\theta, \CI_{r(\af)})$ is a generalized Boolean dynamical system. We call it {\it a generalized Boolean dynamical system associated to $(E, \CL,\CB)$}. 
It then is straightforward to check that 
\[
\{p_A, s_\af p_B: A \in \CB, \af \in \CA ~\text{and}~ B \in \mathcal{I}_{r(\alpha)}\}	
\]
is a Cuntz--Krieger representation of $(\CB, \CA,\theta,\CI_{r(\af)} )$. Then the universal property of $C^*(\CB, \CA,\theta, \CI_{r(\af)})$ gives a $*$-homomorphism 
\[
\phi: C^*(\CB, \CA,\theta, \CI_{r(\af)}) \to C^*(E,\CL,\CB)	
\]
defined by 
\[
\phi(p_A)=p_A ~\text{and}~ \phi(s_{\af,B})=s_\af p_B	
\]
for all $A \in \CB, \af \in \CA$ and $B \in \mathcal{I}_{r(\alpha)}$. Since $s_\af=s_\af p_{r(\af)}$, the family $\{p_A, s_\af p_B: A \in \CB, \af \in \CA ~\text{and}~ B \in \mathcal{I}_{r(\alpha)} \}$ generates $C^*(E, \CL, \CB)$, and hence, the map $\phi$ is onto. Applying the gauge-invariant uniqueness theorem \cite[Corollary 6.2]{CaK2}, we conclude that  $\phi$ is an isomorphism.

We summarize this facts in the following. 
\begin{prop}\label{isom}(\cite[Example 4.2]{CaK2}) Let $(E,\CL,\CB)$ be a  labeled space, where $\CL:E^1 \to \CA$ is onto. Then  $(\CB, \CA,\theta, \CI_{r(\af)})$ is a generalized Boolean dynamical system and 
$C^*(E,\CL,\CB)$ is isomorphic to $C^*(\CB, \CA,\theta, \CI_{r(\af)}).$
\end{prop}

\section{Gauge-invariant ideals of $C^*(E, \CL,\CB)$}  In this section,  we give a complete list of the gauge-invariant ideals of $C^*$-algebras of arbitrary labeled spaces $(E, \CL,\CB)$ and describe the quotients as $C^*$-algebras of relative generalized Boolean dynamical systems.
  Most of results can be obtained  by the same arguments used in \cite{CaK2}. So, we omit its proof.

We  recall  from \cite[Definition 3.4]{JKP} 
that a subset   $\CH$ of $\CB$ is said to be   {\it hereditary} if 
\begin{enumerate}
\item if $A \in \CH$, then $r(A, \af) \in \CH$ for all  $\af \in \CL^*(E)$,
\item if $A, B \in \CH$, then $A \cup B \in \CH$,
\item if $ A \in \CH$ and $B \in \CB$ with $B \subset A$, then $B \in \CH$.
\end{enumerate}
A hereditary set $\CH$ is said to be {\it saturated} if  $A \in H$ 
whenever $A \in \CB_{reg} $ satisfies $r(A,\af) \in \CH$ for all $\af \in \CL^*(E)$ (\cite[Section 2.3]{CaK2}). 

\vskip 1pc 
  
The following lemma shows  how to find a hereditary saturated subset. 
\begin{lem}\label{hereditary saturated set}
For a nonempty set  $A\in \CB$, let
\begin{align}\label{H(A)}\CH(A):=\{B \in \CB: B \subseteq \cup_{i=1}^n r(A, \af_i) ~\text{for some}~  \af_i \in \CL^\#(E) \}. \end{align}
is the smallest hereditary set that contains $A$, and
\begin{align*}
	\CS(\CH(A)):=\{B\in\CB:& \text{there is an }n \geq 0 \text{ such that }r(B, \bt)\in\CH(A)\text{ for all }\beta\in\CL^n(E),\\
	&\text{and}~r(B,\gm)\in\CH(A)\oplus\CB_{reg}\text{ for all }\gamma\in\CL^\#(E) \text{ with }|\gamma|<n\}
\end{align*}
is the smallest saturated hereditary set that contains $A$, where
$$\CH(A)\oplus\CB_{reg}:=\{C \cup D:C \in \CH(A) ~\text{and}~  D \in \CB_{reg}\}.$$
\end{lem}

\begin{proof}
It is straightforward to check that $\CH(A)$ is hereditary , and it is easy to see that if $\CH$ is a hereditary set and $A\in\CH$, then $\CH(A)\subseteq\CH$.

It is rather obvious  that $\CS(\CH(A))$ is hereditary. To show that it is  saturated, let $B \in \CB_{reg}$ satisfy $r(B, \af) \in \CS(\CH(A))$ for all $\af \in \CL^*(E)$.  Put $\CL(BE^1)=\{\af_1, \cdots, \af_n\}$ (this is a finite set since $B \in \CB_{reg}$). Then $r(B, \af_i) \in \CS(\CH(A))$ for each $i$, and thus 
there is an $n_i \geq 1$ such that $r(r(B, \af_i), \bt) \in \CS(\CH(A))$ for all $\bt \in \CL^{n_i}(E)$ and 
$r(r(B, \af_i),\gm)\in\CH(A)\oplus\CB_{reg} $ for all $\gamma\in\CL^\#(E)$ with $|\gamma|<n_i$. 
Take $n := \max_{1 \leq i \leq n}\{n_i\}$. 
Then $r(B, \bt) \in \CH(A)$ for all $\bt \in \CL^{n+1}(E)$ and $r(B, \gm)\in\CH(A)\oplus\CB_{reg} $ for all $\gamma\in\CL^\#(E)$ with $|\gamma|<n+1$. Thus, $B \in \CS(\CH(A))$.
  It is also straightforward to check that if $\CS$ is a saturated hereditary set and $A\in\CS$, then $\CS(\CH(A))\subseteq\CS$.
  
\end{proof}

\begin{ex}\label{exex} For the labeled graph 
 \vskip 1pc
 \hskip 2.5cm
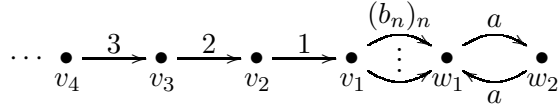
\xy  /r0.3pc/:(-34.2,0)*+{\cdots};(5,0.7)*+{\vdots};
(-30,0)*+{\bullet}="V-3";
(-20,0)*+{\bullet}="V-2";
(-10,0)*+{\bullet}="V-1"; (0,0)*+{\bullet}="V0";
(10,0)*+{\bullet}="V1"; 
(20,0)*+{\bullet}="V2"; 
  "V-3";"V-2"**\crv{(-30,0)&(-20,0)};
 ?>*\dir{>}\POS?(.5)*+!D{};
 "V-2";"V-1"**\crv{(-20,0)&(-10,0)};
 ?>*\dir{>}\POS?(.5)*+!D{};
 "V-1";"V0"**\crv{(-10,0)&(0,0)};
 ?>*\dir{>}\POS?(.5)*+!D{};
 "V0";"V1"**\crv{(0,0)&(5,5)&(10,0)}; ?>*\dir{>}\POS?(.5)*+!D{};
  "V0";"V1"**\crv{(0,0)&(5,-5)&(10,0)}; ?>*\dir{>}\POS?(.5)*+!D{};
   "V1";"V2"**\crv{(10,0)&(15,5)&(20,0)}; ?>*\dir{>}\POS?(.5)*+!D{};
  "V2";"V1"**\crv{(20,0)&(15,-5)&(10,0)}; ?>*\dir{>}\POS?(.5)*+!D{};
   ;(-25,1.5)*+{3};
 (-15,1.5)*+{2};(-5,1.5)*+{1};(5,4.5)*+{(b_n)_{n}};
 (0.1,-2.5)*+{v_1};(10.1,-2.5)*+{w_1};(20.1,-2.5)*+{w_2};
 (-9.9,-2.5)*+{v_{2}};
 (-19.9,-2.5)*+{v_{3}};
 (-29.9,-2.5)*+{v_{4}};
  (15,4)*+{a}; (15,-4)*+{a}; 
\endxy 
\vskip 1pc 

\noindent ,where $v_1$ emits infinitely many labeled edges $(b_n)_{n \geq 1}$, consider the labeled space $(E,\CL,\CE)$. It is clear that   $\CH(r(a))=\{\emptyset, \{w_1\}, \{w_2\}, \{w_1, w_2\}\}$.  
 One sees that $r(\{v_1\}, \bt) \in \CH(r(a))$ for $|\bt| \geq 1$, but $\{v_1\} \notin \CH(r(a)) \oplus \CE_{reg}$. So, $\{v_1\} \notin  \CS(\CH(r(a)))$. It is also easy to see that $\{v_i\} \notin  \CS(\CH(r(a)))$ for each $i > 1$. In fact, $\CS(\CH(r(a)))=\CH(r(a))$. 
  
\end{ex}

\subsection{A quotient Boolean dynamical system $(\CB/ \CH, \CA, \theta)$ associated to $(E,\CL,\CB)$}

Let $(E, \CL,\CB)$ be a labeled space, where $\CL: E^1 \to \CA$ is assumed to be onto. 
If $\CH$ is a hereditary subset of $\CB$, then the relation
\begin{eqnarray}\label{equivalent relation} 
A \sim B \iff  A \cup W = B \cup W~\text{for some}~ W \in \CH
\end{eqnarray}
defines an equivalent relation on $\CB$ (\cite[Proposition 3.6]{JKP}).
We denote the equivalent class of $A \in \CB$ with respect to $\sim$ by $[A]$ (or $[A]_\CH$ if we need to specify the hereditary set $\CH$) and the set of all equivalent classes of $\CB$ by $\CB / \CH$. It is easy to check that $\CB / \CH$ is a Boolean algebra with operations defined by 
\[
[A]\cap [B]=[A\cap B], [A]\cup [B]=[A\cup B] ~\text{and}~ [A]\setminus [B]=[A\setminus B].
\]
The partial order $\subseteq$ on $\CB / \CH$ is characterized by  
\begin{align*} 
[A] \subseteq [B] & \iff A \subseteq B\cup W  ~\text{for some}~  W \in \CH \\
&\iff  [A] \cap [B] =[A].
\end{align*}
If ,in addition, we define $\theta_\af: \CB/\CH \to \CB/\CH$ by $\theta_\af([A])=[r(A, \af)]$ for all $[A] \in \CB/\CH$ and $ \af \in \CA$, then  $(\CB/ \CH, \CA, \theta)$ becomes a Boolean dynamical system (see \cite[Proposition 3.6]{JKP}). We call it a {\em quotient Boolean dynamical system associated to $(E,\CL,\CB)$}.


  \begin{remark} 
Given a labeled space $(E, \CL,\CB)$ and hereditary set $\CH$ of $\CB$, there can be a $[A] \in \CB /\CH$ such that $[A] \neq [\emptyset]$, but $[r(A, \af)]=[\emptyset]$ for all $\af \in \CA$. For example, for the labeled space $(E, \CL, \CE)$ in Example \ref{exex}, 
we have $\CH:=\CH(r(a))$ is a hereditary saturated subset of $\CE$.
It is easy to see that   $[\{v_1\}] \neq [\emptyset]$, but $[r(\{v_1\}, \af)]=[\emptyset]$ for all $\af \in \CA$. 
\end{remark}

 A {\em filter} \cite[Definition 2.6]{COP} $\xi$ in a Boolean algebra $\CB$ is a non-empty subset $\xi \subseteq \CB$  such that 
\begin{enumerate}
\item[$\mathbf{F0}$] $\emptyset \notin \xi$,
\item[$\mathbf{F1}$] if $ A \in \xi$ and  $A \subseteq B$, then $B \in \xi$,
\item[$\mathbf{F2}$] if $A,B \in \xi$, then $A \cap B \in \xi$.
\end{enumerate}
If in addition $\xi$ satisfies 
\begin{enumerate}
\item[$\mathbf{F3}$] if $A \in \xi$  and $B,B' \in \CB$ with $A=B \cup B'$, then either $B \in \xi$ or $B' \in \xi$,
\end{enumerate}
then it is called an {\em ultrafilter} \cite[Definition 2.6]{COP} of $\CB$. A filter is an ultrafilter if and only if it is a maximal element in the set of filters with respect to inclusion. 
We write $\widehat{\CB}$  for the set of all ultrafilters of $\CB$.

\begin{dfn}(\cite[Definition 3.1 and Definition 5.1]{CaK1}) Let  $(E, \CL, \CB)$ be a labeled space and $\alpha\in \CL^*(E) \setminus\{\emptyset\}$ and $\eta\in\widehat{\CB}$. 
\begin{enumerate}
\item  A pair $(\alpha,\eta)$ is  an \emph{ultrafilter cycle} if $r(A, \af)\in\eta$ for all $A \in \eta$.
\item $(E, \CL, \CB)$  satisfies \emph{Condition (K)} if there is no pair $((\alpha,\eta),A)$ where $(\alpha,\eta)$ is an ultrafilter cycle and $A\in\eta$ such that if $\beta\in\CL^*\setminus\{\emptyset\}$, $B \subseteq A$, and $r(B,\bt)\in\eta$, then $B\in\eta$ and $\beta=\alpha^k$ for some $k\in\N$.
\end{enumerate}
\end{dfn}

\begin{lem}\label{LK}(\cite[Theorem 6.3]{CaK1}) Let $(E, \CL,\CB)$ be a labeled space. Then $(E, \CL,\CB)$ satisfies Condition (K) if and only if the quotient Boolean dynamical system $(\CB/ \CH, \CA, \theta)$ associated to $(E, \CL,\CB)$ satisfies Condition (L) for every hereditary saturated subset $\CH$ of $\CB$. 
\end{lem}

\begin{proof}It  follows by \cite[Theorem 6.3]{CaK1}.
\end{proof}

\subsection{Gauge-invariant ideals of $C^*(E, \CL,\CB)$}

Given a hereditary saturated subset $\CH$ of $\CB$, we define an ideal
$$
\CB_{\CH}:=\{ A \in \CB : [A] \in (\CB/\CH)_{reg} \}
$$
of $\CB$.    Choose an  ideal $\CS $  of $\CB_\CH$ (and hence an ideal of $\CB$) such that $\CH \cup \CB_{reg} \subseteq \CS $.  Let $I_{(\CH,\CS)}$ denote the ideal of $C^*(E, \CL, \CB):=C^*(p_A, s_{\af})$ generated by the family of projections 
$$
\biggl\{p_A-\sum_{\af \in \Delta_{[A]}} s_\af p_{r(A, \af)}s_\af^*: A \in \CS\biggr\},
$$
where $\Delta_{[A]}:=\{\af \in \CA: [r(A, \af)] \neq [\emptyset]\}$.
Then the ideal $I_{(\CH,\CS)}$ is gauge-invariant (\cite[Lemma 7.1]{CaK2}) and 
\begin{equation}\label{looks-ideal}
I_{(\CH,\CS)}=\overline{\operatorname{span}}\{s_{\af}(p_{A}-p_{A,\CH})s_{\bt}^*:A \in \CS ~\text{and}~  \af,\bt \in \CL^*(E) \},
\end{equation}
where we put $p_{A, \CH}:=\sum_{\af \in \Delta_{[A]}} s_\af p_{r(A, \af)}s_\af^*$.


We shall prove in Theorem \ref{isomorphism to quotient} that every gauge-invariant ideals of $C^*(E, \CL, \CB)$ is of the form $I_{(\CH,\CS)}$ for a hereditary saturated set $\CH$ and an ideal $\CS$ of $\CB_\CH$ with $\CH \subseteq \CS$ and $\CB_{reg} \subseteq \CS $.
We first observe the following.

\begin{lem}(\cite[Lemma 7.2]{CaK2})\label{her-sat from ideal} 
Let $I$ be a nonzero ideal in $C^*(E, \CL, \CB)$.
\begin{enumerate}
\item  The set $\CH_I:=\{A \in \CB: p_A \in I\}$
is a hereditary saturated subset of $\CB$.
\item  The set 
 $$
\CS_I:=\biggl\{A \in \CB_{\CH_I}: p_A-\sum_{\af \in \Delta_{[A]}} s_\af p_{r(A, \af)}s_\af^* \in I\biggr\}	
$$
is an ideal of $\CB_{\CH_I}$ (and hence an ideal of $\CB$) with $\CH_I \subseteq \CS_I$ and $\CB_{reg} \subseteq \CS_I$. 
\end{enumerate} 
\end{lem}

\begin{proof} It  follows by Proposition \ref{isom} and \cite[Lemma 7.2]{CaK2}.
\end{proof}

 A quotient labeled space was  introduced in \cite{JKP} to 
study the ideal structure of  $C^*$-algebras of a set-finite and receiver set-finite labeled space $(E, \CL,\CB)$ with  $E$ having no sinks. Given a hereditary saturated subset $\CH$ of $\CB$, let $I_\CH$ be the ideal of $C^*(E, \CL,\CB)$ generated by the projections $\{p_A : A \in \CH\}$. 
Then the quotient algebra $C^*(E, \CL,\CB)/I_\CH$ by the gauge-invariant ideal $I_\CH$ 
as a $C^*$-algebras $C^*(E, \CL, \CB/\CH)$ of a quotient labeled space $(E, \CL, \CB/\CH)$ (\cite[Theorem 5.2]{JKP}). 
The quotient labeled space $(E, \CL, \CB/\CH)$   is nothing but a Boolean dynamical system $(\CB/\CH, \CA, \theta)$. So, to generalize this result to $C^*$-algebras of arbitrary labeled spaces,  we  use   relative generalized Boolean dynamical systems rather than  we newly define   relative quotient labeled spaces  of arbitrary labeled spaces.

\begin{prop}(\cite[Proposition 7.3]{CaK2})\label{isomorphism to quotient} 
Let $(E, \CL,\CB)$ be a labeled space. Suppose that $I$ is an ideal of $C^*(E, \CL,\CB)$. 
There is then a surjective $*$-homomorphism
\[
\phi_I:C^*(\CB/\CH_I,\CA,\theta,[\CI_{r(\af)}];[\CS_I])\to
C^*(E, \CL,\CB)/I
\] 
such that
\[
\phi_I(p_{[A]})=p_A+I ~\text{ and}~	
\phi_I(s_{\alpha, [B]})=s_{\alpha}p_B+I,
\]
 where $[\CI_{r(\af)}]:=\{[A] \in \CB / \CH_I: A \in \CI_{r(\af)} \}$ and $[\CS_I]:=\{[A] \in \CB/ \CH_I:  A \in \CS_I\}$.
 Moreover, the following are equivalent.
\begin{enumerate}
\item $I$ is gauge-invariant.
\item The map $\phi_I$ is an isomorphism.
\item $I=I_{(\CH_I,\CS_I)}$.
\end{enumerate}
\end{prop}

\begin{proof}It  follows by Proposition \ref{isom} and \cite[Proposition 7.3]{CaK2}.
\end{proof}
We further say that the map $(\CH,\CS) \mapsto I_{(\CH,\CS)} $ is a lattice isomorphism. 
The set of pairs $(\CH,\CS)$, where $\CH$ is a hereditary saturated subset of $\CB$ and $\CS$ is an ideal of $\CB_\CH$ with $\CH \cup \CB_{reg}\subseteq \CS $ is a lattice with respect to the order relation defined by $(\CH_1,\CS_1)\le (\CH_2,\CS_2)\iff (\CH_1\subseteq \CH_2 ~\text{and}~ \CS_1\subseteq\CS_2)$. The set of gauge-invariant ideals of $C^*(E, \CL, \CB)$ is a lattice with the order given by set inclusion.


\begin{thm}\label{gauge invariant ideal:characterization}(\cite[Theorem 7.4]{CaK2}) Let  $(E, \CL, \CB)$ be a labeled space.
Then the map $(\CH,\CS) \mapsto I_{(\CH,\CS)} $ is a lattice isomorphism between the lattice of all pairs $(\CH,\CS)$, where $\CH$ is a hereditary  saturated subset of $\CB$ and $\CS$ is an ideal of $\CB_\CH$ with $\CH \cup \CB_{reg} \subseteq \CS $,  and the lattice of all  gauge-invariant ideals of $C^*(E, \CL,\CB)$.
\end{thm}

\begin{proof}It  follows by Proposition \ref{isom} and \cite[Theorem 7.4]{CaK2}.
\end{proof}

\begin{ex} Let $(E, \CL)$ be the following labeled graph
 \vskip 1pc \hskip 9pc 
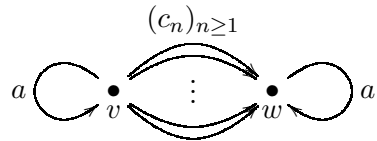
\xy /r0.5pc/:
 (0,0)*+{\bullet}="V1";
 (10,0)*+{\bullet}="V2";
  "V1";"V1"**\crv{(0,0)&(-3, 3)&(-6,0)&(-3,-3)&(0,0)};?>*\dir{>}\POS?(.5)*+!D{};
  "V2";"V2"**\crv{(10,0)&(13, 3)&(16,0)&(13,-3)&(10,0)};?>*\dir{>}\POS?(.5)*+!D{};
   "V1";"V2"**\crv{(0,0)&(5,6)&(10,0)};?>*\dir{>}\POS?(.5)*+!D{};
    "V1";"V2"**\crv{(0,0)&(5,4.5)&(10,0)};?>*\dir{>}\POS?(.5)*+!D{};
     "V1";"V2"**\crv{(0,0)&(5,-4.5)&(10,0)};?>*\dir{>}\POS?(.5)*+!D{};
   "V1";"V2"**\crv{(0,0)&(5,-6)&(10,0)};?>*\dir{>}\POS?(.5)*+!D{};
 (0,-1.3)*+{v},(10,-1.3)*+{w}, (-6,0)*+{a}, (16,0)*+{a},(5,0.5)*+{\vdots},(5,4.5)*+{(c_n)_{n \geq 1}},
 \endxy
 \vskip 1pc
 \noindent ,where $v$ emits infinitely many labeled edges $(c_n)_{n \geq 1}$. 
 Then $\CE=\{\emptyset, \{v\}, \{w\}, \{v,w\}\}$, $\CE_{reg}=\{\emptyset, \{w\}\}$, and 
 $\CH(\{w\})=\{\emptyset, \{w\}\}$. Since $r(\{v\}, a^n) \notin \CH(\{w\})$ for all $n \geq 1$, it follows that $\{v\} \notin \CS(\CH(\{w\}))$. Thus, $\CS(\CH(\{w\}))=\{\emptyset, \{w\}\}$.
 
Put $\CH:=\CS(\CH(\{w\}))=\{\emptyset, \{w\}\}$. Then $\CE/\CH=\{[\emptyset], [\{v\}]\}=(\CE/\CH)_{reg}$ and 
\begin{align}\label{BH}
\CE_{\CH}=\{ A \in \CE : [A] \in (\CE/\CH)_{reg} \}=\{\emptyset, \{v\}, \{w\}, \{v,w\}\}=\CE.
\end{align}
Let $I_{(\CH, \CE_{\CH})}$ denote the only one gauge-invariant ideal of $C^*(E, \CL, \CE):=C^*(p_A, s_{\af})$ generated by the family of projections 
$$
\biggl\{p_A-\sum_{\af \in \Delta_{[A]}} s_\af p_{r(A, \af)}s_\af^*: A \in \CE_{\CH}\biggr\}.	
$$
Note that $[\CE_\CH]=\{[\emptyset], [\{v\}]\}=(\CE/\CH)_{reg}$. Then, by Proposition \ref{isomorphism to quotient}, we have that  
 that $$C^*(E,\CL,\CE)/I_{(\CH, \CE_{\CH})} \cong C^*(\CE / \CH, \CA, \theta, [\CI_{r(\af)}]).$$  
Since  $C^*(\CE / \CH, \CA, \theta, [\CI_{r(\af)}]) $ is generated by  $\{ p_{[\{v\}]}, s_{a, [\{v\}]}\}$ satisfying  $s_{a, [\{v\}]}^*s_{a, [\{v\}]}=p_{[\{v\}]}=s_{a, [\{v\}]}s_{a, [\{v\}]}^*$, we conclude that 
 $C^*(E,\CL,\CE)/I_{(\CH, \CE_{\CH})}$ is isomorphic to 
 $ C(\T)$.  
\end{ex}

\section{Simple labeled graph $C^*$-algebras} In this section, we introduce strong cofinality, minimality and disagreeablity for an arbitrary labeled space $(E, \CL,\CE)$, and state our main result, Theorem \ref{equiv:simple}.
\subsection{Strong cofinality and minimality} 
 We  first recall the notion of strong cofinality given \cite[Definition 3.1]{JK}. A set-finite and receiver set-finite labeled space $(E, \CL,\CE)$ with  $E$ having no sinks or sources
was called strongly cofinal if for all $x \in \CL(E^\infty)$, $[v]_l \in \CE$ and $w \in s(x)$, there are $N \geq 1$ and a finite number of paths $\ld_1, \cdots, \ld_m \in \CL^*(E)$
such that $$r([w]_1, x_{[1,N]}) \subset \cup_{i=1}^m r([v]_l, \ld_i).$$

It is shown in \cite[Theorem 3.16]{JK} that if $(E, \CL,\CE)$ is strongly cofinal in the above sence and is disagreeable, then $C^*(E, \CL,\CE)$ is simple. 
But, the following example shows that  there is a strongly cofinal labeled space in the sense of \cite[Definition 3.1]{JK} and is disagreeable, but its associated $C^*$-algebra is not simple.

\begin{ex}\label{counter ex} Consider the following labeled graph $(E, \CL)$:
\vskip 1pc 
 \hskip 1.1cm
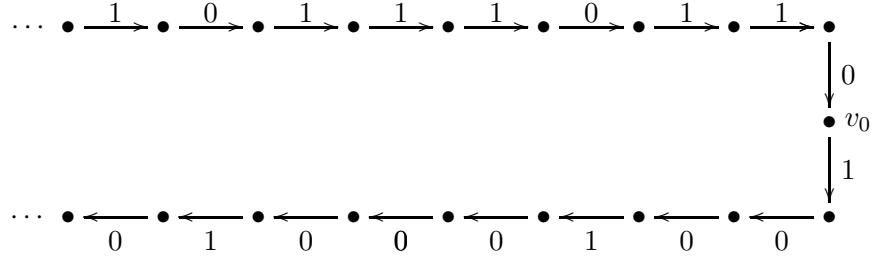
\xy  /r0.3pc/:(-44.2,0)*+{\cdots};(-44.3,-20)*+{\cdots };
(-40,0)*+{\bullet}="V-4";
(-30,0)*+{\bullet}="V-3";
(-20,0)*+{\bullet}="V-2";
(-10,0)*+{\bullet}="V-1"; 
(0,0)*+{\bullet}="V0";
(10,0)*+{\bullet}="V1"; 
(20,0)*+{\bullet}="V2";
(30,0)*+{\bullet}="V3";
(40,0)*+{\bullet}="V4";
(40,-10)*+{\bullet}="W1";
(40,-20)*+{\bullet}="W2";
(30,-20)*+{\bullet}="W3";
(20,-20)*+{\bullet}="W4";
(10,-20)*+{\bullet}="W5";
(0,-20)*+{\bullet}="W6";
(-10,-20)*+{\bullet}="W7";
(-20,-20)*+{\bullet}="W8";
(-30,-20)*+{\bullet}="W9";
(-40,-20)*+{\bullet}="W10";
 "V-4";"V-3"**\crv{(-40,0)&(-30,0)}; ?>*\dir{>}\POS?(.5)*+!D{};
 "V-3";"V-2"**\crv{(-30,0)&(-20,0)}; ?>*\dir{>}\POS?(.5)*+!D{};
 "V-2";"V-1"**\crv{(-20,0)&(-10,0)}; ?>*\dir{>}\POS?(.5)*+!D{};
 "V-1";"V0"**\crv{(-10,0)&(0,0)}; ?>*\dir{>}\POS?(.5)*+!D{};
 "V0";"V1"**\crv{(0,0)&(10,0)}; ?>*\dir{>}\POS?(.5)*+!D{};
 "V1";"V2"**\crv{(10,0)&(20,0)}; ?>*\dir{>}\POS?(.5)*+!D{};
 "V2";"V3"**\crv{(20,0)&(30,0)}; ?>*\dir{>}\POS?(.5)*+!D{};
 "V3";"V4"**\crv{(30,0)&(40,0)}; ?>*\dir{>}\POS?(.5)*+!D{};
 "V3";"V4"**\crv{(30,0)&(40,0)}; ?>*\dir{>}\POS?(.5)*+!D{};
"V4";"W1"**\crv{(40,0)&(40,-10)}; ?>*\dir{>}\POS?(.5)*+!D{};
"W1";"W2"**\crv{(40,-10)&(40,-20)}; ?>*\dir{>}\POS?(.5)*+!D{};
"W2";"W3"**\crv{(40,-20)&(30,-20)}; ?>*\dir{>}\POS?(.5)*+!D{};
"W3";"W4"**\crv{(30,-20)&(20,-20)}; ?>*\dir{>}\POS?(.5)*+!D{};
"W4";"W5"**\crv{(20,-20)&(10,-20)}; ?>*\dir{>}\POS?(.5)*+!D{};
"W5";"W6"**\crv{(10,-20)&(0,-20)}; ?>*\dir{>}\POS?(.5)*+!D{};
"W6";"W7"**\crv{(0,-20)&(-10,-20)}; ?>*\dir{>}\POS?(.5)*+!D{};
"W7";"W8"**\crv{(-10,-20)&(-20,-20)}; ?>*\dir{>}\POS?(.5)*+!D{};
"W8";"W9"**\crv{(-20,-20)&(-30,-20)}; ?>*\dir{>}\POS?(.5)*+!D{};
"W9";"W10"**\crv{(-30,-20)&(-40,-20)}; ?>*\dir{>}\POS?(.5)*+!D{};
 (-35,1.5)*+{1};(-25,1.5)*+{0};
 (-15,1.5)*+{1};(-5,1.5)*+{1};(5,1.5)*+{1};
 (15,1.5)*+{0};(25,1.5)*+{1};(35,1.5)*+{1};
(42,-5)*+{0};(42,-15)*+{1};(43,-10)*+{v_0};
(35,-22.5)*+{0};(25,-22.5)*+{0};
(15,-22.5)*+{1};(5,-22.5)*+{0};(-5,-22.5)*+{0};(-5,-22.5)*+{0};(-15,-22.5)*+{0};
(-25,-22.5)*+{1};(-35,-22.5)*+{0};
\endxy 
\vskip 1pc 
\noindent 
Then $(E, \CL,\CE)$ is set-finite, receiver set-finite and left-resolving.   For every vertex $v \in E^0$, there is $l > 0$ and $\af \in \CL^{l}(E)$ such that $r(\af)=[v]_{l}=\{v\} \in \CE$.
Thus, one sees that $\CE=\{ \cup_{i=1}^n r(\af_i): \af_i \in \CL^*(E) ~\text{and}~ n \in \N \}.$
It then easy to see that $(E, \CL,\CE)$ is strongly cofinal in the sense of  \cite[Definition 3.1]{JK}. 

On the other hand, consider the smallest hereditary saturated set $\CH$ containing $r(1)$. One can see that $r(0) \notin \CH$. So, $\CH$ is a proper hereditary saturated subset of $\CE$. Thus, by \cite[Theorem 5.2]{JKP}, there is a nontrivial ideal of $C^*(E, \CL,\CE)$, so that $C^*(E, \CL,\CE)$ is not simple. In fact, $C^*(E, \CL,\CE)$ contains many ideals that is not gauge-invariant. 
To see this, note that for each $\af \in \CL^*(E)$, if $\af=1\af'$ or $\af=\af'1$ for some $\af'$, then $r(\af) \in \CH$. Observe also that $r(0) \sim r(0^n)$ for each $n \in \N$ since  $r(0) \cup \big(r(10)\cup \cdots \cup r(10^{n-1}) \big)  =r(0^n) \cup  \big(r(10)\cup \cdots \cup r(10^{n-1}) \big)$ for each $n \in \N$.  So, we have $\CE/ \CH=\{[r(0)], [\emptyset]\}$.
 One then can see that $C^*(E,\CL,\CE)/I_{\CH} \cong C^*(\CE/\CH, \CA, \theta) \cong C(\T)$, where $\CA=\{0,1\}$ Thus, $C^*(E, \CL,\CE)$ contains many ideals that is not gauge-invariant. 
  \end{ex}

As we see in the above example, we need to modify the definition of strong cofinality given in \cite[Definition 3.1]{JK}. To do that, let $$\overline{\CL(E^\infty)}:=\{x\in \CA^{\mathbb N}\mid 
x_{[1,n]}\in \CL(E^n) \ \text{for all } n\geq 1\}$$ 
be the set of all right infinite sequences $x$ such that 
all of its subpaths occurs as a labeled path in $(E,\CL)$.
Clearly $\CL(E^\infty)\subset \overline{\CL(E^\infty)}$, and 
in fact, $\overline{\CL(E^\infty)}$ is the closure of $\CL(E^\infty)$ in the 
totally disconnected perfect space $\CA^{\mathbb N}$ 
which has the topology with  a countable 
basis of open-closed cylinder sets 
$Z(\af):=\{x\in \CA^{\mathbb N}: x_{[1,n]}=\af\}$, $\af\in \CA^n$, $n\geq 1$ 
(see Section 7.2 of \cite{Ki}).  
 In general, $\CL(E^\infty)\subsetneq \overline{\CL(E^\infty)}$. For example, in Example \ref{counter ex}, we see that $1^{\infty}, 0^{\infty} \notin \CL(E^\infty)$, but  $ 1^{\infty}, 0^{\infty}  \in \overline{\CL(E^\infty)}$. 

Adopting \cite[Definition 2.10]{JP2018}, we introduce  strong cofinality of an arbitrary labeled space.  
\begin{dfn}\label{strongly cofinal}
 We say that a labeled space $(E,\CL,\CE)$ is {\it strongly cofinal} if for each nonempty set $A\in \CE $ and   $x\in \overline{\CL(E^\infty)}$, there exist $N \in \N$ and a finite number of paths 
$\ld_1, \dots, \ld_m\in \CL^*(E)$ such that 
$$r(x_{[1,N]})\subseteq \cup_{i=1}^m r(A,\ld_i).$$ 
\end{dfn}
Throughout the paper if we mention strong cofinality, we mean  Definition  \ref{strongly cofinal}. 
 
 \begin{ex} We continue Example \ref{counter ex}. 
For $x:=0^{\infty} \in   \overline{\CL(E^\infty)} \setminus \CL(E^\infty)$ and $\{v_0\} \in \CE$, we see that
$r(x_{[1,n]})=r(0^n) \nsubseteq \cup_{i=1}^m r(\{v_0\},\ld_i) $
for any paths 
 $\ld_1, \dots, \ld_m\in \CL^*(E)$ and any $n \in \N$.
 Thus, $(E, \CL,\CE)$ is not strongly cofinal.
\end{ex}

\begin{remark}\label{sc:remark} Let $(E,\CL,\CE)$ be a labeled space.
\begin{enumerate}
\item  $(E,\CL,\CE)$ is strongly cofinal if and only if  for all $\emptyset \neq A \in \CE$, $B \in  \CE $ and $x \in \overline{\CL(E^\infty)}$, there exists 
an $N \geq 1$  such that 
$r(B, x_{[1,N]}) \in \CH(A).$ 
\item If $E$ has no sources and $(E,\CL, \CE)$ is set-finite and receiver set-finite, then  $(E,\CL, \CE)$ is strongly cofinal if and only if   for each $[v]_l\in \CE $, $x\in \overline{\CL(E^\infty)}$ and $w \in s(x)$, there exist an $N \geq 1$ and a finite number of paths 
$\ld_1, \dots, \ld_m\in \CL^*(E)$ such that 
$r([w]_1, x_{[1,N]})\subseteq \cup_{i=1}^m r([v]_l,\ld_i).$ 
\end{enumerate}
\end{remark}

\begin{proof}(1)($\Rightarrow$) It is clear.

($\Leftarrow$) Let $A \in  \CE$ and $x=x_1x_2 \cdots \in \overline{\CL(E^\infty)} $. Then for $r(x_1) \in \CE$ and $x_2x_3 \cdots \in \overline{\CL(E^\infty)} $, there exist $N \geq 1$ and a finite number of paths 
$\ld_1, \dots, \ld_m\in \CL^*(E)$ such that
$$r(x_1 \cdots x_N)=r(r(x_1), x_2x_3 \cdots x_N) \subseteq \cup_{i=1}^m r(A,\ld_i).$$ 

\noindent
(2) ($\Rightarrow$) It is clear since $r([w]_1, x_{[1,N]})\subseteq r(x_{[1,N]})$ for any  $w \in s(x)$.

($\Leftarrow$)
Let $A \in \CE$ and $x=x_1x_2 \cdots \in \overline{\CL(E^\infty)}$. We may assume that $A=[v]_l$ for some $v \in E^0$. Since $(E, \CL,\CE)$ is receiver set-finite, $r(x_1) = \cup_{i=1}^k [w_i]_1$. Thus  
if $[w_i]_1 \cap s(x_2)  \neq \emptyset$ for $i \in \{1, \cdots, k\} $, there are $N_i \geq 1$ and $\ld^i_1, \cdots, \ld^i_{m_i} \in \CL^*(E)$ such that
$$r([w_i]_1,x_2 \cdots x_{N_i} ) \subset \cup_{j=1}^{m_i} r(A, \ld^i_j).$$
Choose $N=\max \{N_i : [w_i]_1 \cap s(x_2) \neq \emptyset \}$. Then for all $i$ with $[w_i]_1 \cap s(x_2) \neq \emptyset$, 
$$r([w_i]_1,x_2 \cdots x_{N} ) \subset \cup_{j=1}^{m_i} r(A, \bt^i_j),$$
where $\bt_j^i=\ld_j^i(\bt_j^i)'$. Thus 
$$r(x_1 \cdots x_{N} )=r(r(x_1), x_2 \cdots x_N) \subset \cup_{i=1}^k r([w_i]_1,x_1 \cdots x_{N} ) \subset \cup _{i=1}^k\cup_{j=1}^{m_i} r(A, \bt^i_j).$$
\end{proof}

\begin{dfn}\label{minimal} We say that $(E, \CL, \CE)$ is {\it minimal} if $\{\emptyset\}$ and $\CE$ are the only hereditary saturated subsets of $\CE$. 
\end{dfn}

 We now describe a number of equivalent conditions of minimality.   The ideal structure of $C^*(E,\CL,\CE)$ will be  used to prove theorem.  
 \begin{thm} \label{equiv:minimal}
Let $(E,\CL,\CE)$ be a labeled space. The following are equivalent.
\begin{enumerate}
\item $(E,\CL,\CE)$ is minimal. 
\item $\CS(\CH(A))=\CE$ for every $A\in\CB\setminus\{\emptyset\}$.
\item $(E,\CL,\CE)$ is strongly cofinal and for any  $ A \in \CE \setminus \{\emptyset\}$ and $B \in \CE$, there exists  $C \in \CE_{reg}$ such that  $B \setminus C\in \CH(A)$. 
\item The only ideal of $C^*(E,\CL,\CE)$ containing $p_A$ for some $A \in \CE \setminus \{\emptyset\}$ is $C^*(E,\CL,\CE)$.
\item The only non-zero ideal of $C^*(E,\CL,\CE)$ which is gauge-invariant is $C^*(E,\CL,\CE)$.
 \end{enumerate}
 \end{thm}

\begin{proof} (1)$\implies$(2): It is obvious.

(2)$\implies$(3):  Choose $\emptyset \neq A \in \CE$, $B \in \CE$ and $x=x_1x_2 \cdots \in \overline{\CL(E^\infty)}$.  Then $B\in\CS(\CH(A))$. It follows from the description of $\CS(\CH(A))$ givne in Lemma~\ref{hereditary saturated set} that there is an $n\geq 0$ such that $r(B, \bt)\in\CH(A)$ for all $\beta\in\CL^n(E)$, and $r(B, \gm)\in\CH(A)\oplus\CE_\reg$ for all $\gamma\in\CL^\#(E)$ with $|\gamma|<n$. If $n=0$ and we let $C=\emptyset$, then $C\in\CE_\reg$, $B\setminus C=B\in\CH(A)$. If $n>0$, then $r(B, x_{[1,n]})\in\CH(A)$ and there is a $C\in\CE_\reg$ such that $B\setminus C\in\CH(A)$. Thus, (3) holds (see Remark \ref{sc:remark}(1)).

 (3)$\implies$(1):
Choose a nonempty hereditary saturated subset $\CH$ of $\CE$.  Take $\emptyset \neq A \in \CH$.  Suppose that  $B \notin \CH$ for a set $B \in \CE$. 
Then there is a $C_1 \in \CE_{reg}$ such that $B \setminus C_1 \in \CH(A) (\subset \CH)$. Since $B \notin \CH$, we have $C_1 \notin \CH$. So,
 there is $x_1 \in \CL(E^1)$  such that $r(C_1, x_1) \notin \CH$ since $\CH$ is saturated. 
 We can then choose $C\in\CE_\reg$ such that $r(C_1, x_1)\setminus C\in\CH(A)$.
  Let $C_2:=C\cap r(C_1, x_1)$. Since $r(C_1, x_1) \notin\CH$, it follows that $C_2\notin\CH$. Since $C_2\in\CE_\reg$, we deduce that there is an $x_2\in\CL(E^1)$ such that
      $r(C_2, x_2)\notin\CH$.
       Continuing  this process, we can construct a sequence $(C_n,x_n)_{n\in\N}$ such that $C_n\in\CE_\reg\setminus\CH$, $ x_n\in\CL(E^1)$, $C_{n+1}\subseteq r(C_n, x_n),$ and $r(C_n, x_n)\setminus C_{n+1}\in\CH(A)$  for each $n\in\N$. Let $x:=x_1x_2\cdots   \in \overline{\CL(E^\infty)}$.   Then $C_{n+1}\subseteq  r(C_1, x_{[1,n]})$ for each $n \geq 1$. 
        Since $C_{n+1} \notin \CH$ for all $n \in \N$,  $r(C_1, x_{[1,n]}) \notin \CH$ for all $n \in \N$.    Thus, $r(x_{[1,n]}) \notin \CH(A)$ for each $n \geq 1$, which  contradicts to strong cofinality of $(E, \CL,\CE)$. Thus, we conclude that $\CH=\CE$.

 (1)$\implies$(4):   Let $I$ be an ideal of $C^*(E,\CL,\CE)$ such that $p_A \in I$ for some $\emptyset \neq A \in \CE$. Then $\CH_I=\{A \in \CE: p_A \in I \}$ is a nonempty saturated hereditary subset of $\CE$ (see \cite[Lemma 7.2]{CaK2}). Then $\CH_I=\CE$ by assumption. Thus, $I=C^*(E,\CL,\CE).$
      
   (4)$\implies$(5):      Let $I$ be a nonzero gauge-invariant ideal of $C^*(E,\CL,\CE)$. Then the set $\CH_I=\{A \in \CE: p_A \in I\}$ is nonempty (\cite{JKP}). It means  that $I$ contains a vertex projection. So, $I=C^*(E,\CL,\CE)$.
      
      (5)$\implies$(1): It follows by Theorem \ref{gauge invariant ideal:characterization}.        
 \end{proof}

\begin{prop}\label{suff:minimal} If $(E, \CL,\CE)$ is strongly cofinal and
 for any  $ A \in \CE \setminus \{\emptyset\}$ and $B \in \CE \setminus \CE_{reg}$, we have $B \in \CH(A)$, then  $(E, \CL,\CE)$ is minimal. 
\end{prop}

\begin{proof} Choose a nonempty hereditary saturated subset $\CH$ of $\CE$.  Take $\emptyset \neq A \in \CH$.  Note first that if $B \notin \CH$ for a set $B \in \CE$, $B\in \CE_{reg}$; if  $B \in \CE \setminus \CE_{reg}$, then $B \in \CH(A) \subset \CH  $ by assumption, a contradiction. So,  if $B \notin \CH$ for a set $B \in \CE_{reg}$,  there is $x_1 \in \CL(E^1)$  such that $r(B, x_1) \notin \CH$ since $\CH$ is saturated. If $r(B, x_1) \in \CE \setminus \CE_{reg}$, then again $r(B, x_1) \in \CH$, a contradiction.  So, $r(B, x_1) \in \CE_{reg}$. Repeating this process, we see that  there is an infinite path $x=x_1x_2 \cdots \in \overline{\CL(E^\infty)}$ such that $r(B, x_{[1,n]}) \in \CE_{reg} \setminus \CH$ for all $n \in \N$.
 Thus, $r(x_{[1,n]}) \notin \CH$ for all $n \in \N$ since $r(B,x_{[1,n]} ) \subset r(x_{[1,n]})$. 
  It then  contradicts to strong cofinality of $(E, \CL,\CE)$. Thus, we conclude that $\CH=\CE$.
\end{proof}
The converse of Proposition \ref{suff:minimal} is not true. 

\begin{ex} Consider the following labeled graph $(E, \CL)$:
\vskip 1pc \hskip 12pc 
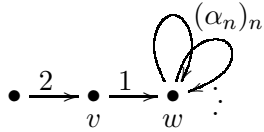
\xy /r0.5pc/:
 (-10,0)*+{\bullet}="V-2";
  (-5,0)*+{\bullet}="V-1";
 (0,0)*+{\bullet}="V1"; 
 "V-2";"V-1"**\crv{(-10,0)&(-5,0)}; ?>*\dir{>}\POS?(.5)*+!D{};
 "V-1";"V1"**\crv{(-5,0)&(0,0)}; ?>*\dir{>}\POS?(.5)*+!D{};
 "V1";"V1"**\crv{(0,0)&(-2,3)&(0,6.2)&(2,3)&(0,0)};?>*\dir{>}\POS?(.5)*+!D{};
"V1";"V1"**\crv{(0,0)&(0.5, 3.4)&(4,4)&(3.8,1)&(0,0)};?>*\dir{>}\POS?(.5)*+!D{};
  (2.7, -1.1)*+{.},(3,-0.3)*+{.},(2.7,0.5)*+{.}, (-5, -1.5)*+{v}, (0, -1.5)*+{w},
  (-8, 1)*+{2}, (-3, 1)*+{1}, (3.5,5)*+{(\af_n)_n}
 \endxy
\vskip 1pc

\noindent ,where $w$ emits infinitely many labeled edges $(\af_n)_{n \geq 1}$.
Then  $\CE=\{ \emptyset, \{v\}, \{w\}, \{v,w\}\} $ and $\CE_{reg}=\{\emptyset, \{v\}\}$.
It is rather obvious that $\CS(\CH(\{v\}))=\CE$ and $\CS(\CH(\{v, w\}))=\CE$. Consider $ \CH(\{w\})$. Then $\CH(\{w\})=\{\emptyset, \{w\}\}$ and  clearly $\{v\} \in \CS(\CH(\{w\}))$. Since $r(\{v,w\}, \bt) \in \CH(\{w\})$ for all $\bt \in \CL(E^1)$ and $\{v, w\}=\{v\} \cup \{w\} \in \CE_{reg} \oplus \CH(\{w\})$, we have  $\{v, w\} \in \CS(\CH(\{w\}))$. Thus, $\CS(\CH(\{w\}))=\CE$. 
Thus,  $(E, \CL, \CE)$ is minimal by Theorem \ref{equiv:minimal}. It is also easy to see that $(E, \CL, \CE)$ is strongly cofinal. But,  $\{v,w\}  \notin \CH(\{w\})$.

\end{ex}

 \begin{cor}\label{cor:minimal} If $E$ has no sinks and $(E,\CL,\CE)$ is set-finite, then 
 $(E,\CL,\CE)$ is minimal if and only if $(E,\CL,\CE)$ is strongly cofinal. 
\end{cor}

\begin{proof}We only need to show "if" part: Note that $\CE=\CE_{reg}$. Choose $\CH$ is a nonempty hereditary saturated subset of $\CE$. If 
$B \notin \CH$ for a set $B \in \CE$,  
 there is an infinite path $x=x_1x_2 \cdots \in \overline{\CL(E^\infty)}$ such that $r(B, x_{[1,n]}) \notin \CH$ for all $n \in \N$ since $\CH$ is saturated.
 Thus, $r(x_{[1,n]}) \notin \CH$ for all $n \in \N$ since $r(B,x_{[1,n]} ) \subset r(x_{[1,n]})$. 
   It  contradicts to strong cofinality of $(E, \CL,\CE)$. Thus, $\CH=\CE$. 
\end{proof}

The following lemma is proved in \cite[Lemma 3.6]{JP2018} for a set-finite and receiver set-finite  labeled space $(E, \CL,\CE)$  with  $E$ having no sinks or sources under the assumption that  $C^*(E,\CL,\CE)$ is simple. But, to prove it they only use  minimality of   $(E, \CL,\CE)$ as a  property of   simplicity of   $C^*(E,\CL,\CE)$. So, we can weaken the assumption as follows. 
The idea of proof is same with \cite[Lemma 3.6]{JP2018}. We have only included the proof of  Lemma \ref{loop}(1) for completeness. 

For a path $\bt:=\bt_1 \cdots \bt_{|\bt|}$, let $\bar{\bt}=\bt\bt\bt \cdots$ denotes the 
infinite repetition of $\bt$ (\cite[Notation 3.5]{JP2018}). we call a path $\bt \in \CL^*(E)$ {\it irreducible}
 if it is not a repetition of its proper initial path. 

\begin{lem}\label{loop}(\cite[Lemma 3.6]{JP2018}) Let $(E,\CL,\CE)$ be a minimal labeled space. If there are a nonempty set $A_0 \in \CE$ and an irreducible path $\bt \in \CL^*$ such that $$\CL(A_0E^{n|\bt|})=\{\bt^n\}$$ for all $n \geq 1$, then the following hold.
\begin{enumerate}
\item There is an $N \geq 1$ such that for all $n \geq N$, 
$$r(A_0, \bar{\bt}_{[1,n]}) \subset \cup_{j=1}^{n-1}r(A_0, \bar{\bt}_{[1,j]}).$$
\item There is an $N \geq 1$ such that for all $k \geq 1$, 
$$r(A_0, \bar{\bt}_{[1,N+k]}) \subset \cup_{j=1}^{N}r(A_0, \bar{\bt}_{[1,j]}).$$
\item There is an $N_0 \geq 1$ such that for all $k \geq 1$, 
\begin{align*}r(A_0, \bt^{N+k}) \subset \cup_{j=1}^{N_0} r(A_0, \bt^j).
\end{align*}
Moreover, $A=r(A,\bt)$ for $A:=\cup_{i=1}^{N_0} r(A_0,\bt^j)$.
\end{enumerate}
\end{lem}

\begin{proof}(1): Note first that 
\begin{align}\label{mod} r(A_0, \bar{\bt}_{[1,j]}) \cap r(A_0, \bar{\bt}_{[1,k]}) \neq \emptyset \implies j=k ~(\text{mod} |\bt|).
\end{align}
Assume to the contrary that 
$r(A_0, \bar{\bt}_{[1,n]}) \setminus \cup_{j=1}^{n-1}r(A_0, \bar{\bt}_{[1,j]}) \neq \emptyset$
for infinitely many $n \geq 1$. Then by (\ref{mod}), 
 \begin{align}\label{assumption} r(A_0, \bt^n) \not\subset\cup_{j=1}^{n-1}r(A_0, \bt^j) \end{align}
for infinitely many $n \geq 1$. We then claim that $r(\bt^r) \notin \CH(A_0)$ for all $r \geq 1$. If $r(\bt^r) \in \CH(A_0)$ for some $r \geq1$, then $r(\bt^r)  \subset \cup_{i=1}^k r(A_0, \bar{\bt}_{[1,m_i]})$  for some $m_1, \cdots, m_k \geq 1$. Then, for each $i$, $m_i=k_i|\bt|$ for some $k_i \geq 1$. Take $m:=\max_i\{k_i\}$. Then we have $r(\bt^r) \subset \cup_{i=1}^m r(A_0,\bt^i)$.  But, then for all sufficiently large $n > m|\bt|$, 
$$r(A_0, \bt^n) = r(r(A_0, \bt^{n-r}), \bt^r) \subset r(\bt^r) \subset \cup_{i=1}^m r(A_0,\bt^i) \subset \cup_{j=1}^{n-1}r(A_0, \bar{\bt}_{[1,j]}),$$ which contradicts to (\ref{assumption}). Thus, $r(\bt^r) \notin \CH(A_0)$ for all $r \geq 1$. It then follows that $r(\bt^r) \notin \CS(\CH(A_0))$ for all $r \geq 1$. But, this is not the case since the minimlity of $(E, \CL,\CE)$ implies $\CS(\CH(A_0))=\CE$. 

(2): Follows by \cite[Lemma 3.6 (i)]{JP2018}.

(3): Follows by \cite[Lemma 3.6 (ii)]{JP2018}.
\end{proof}

\begin{remark} For a set-finite and receiver set-finite labeled space $(E, \CL,\CE)$ with  $E$ having no sinks or sources, it is shown in  \cite[Theorem 3.7]{JP2018} that  if $C^*(E, \CL,\CE)$ is simple, then $(E, \CL,\CE)$ is disagreeable. Thus, if $C^*(E, \CL,\CE)$ is simple, the labeled space $(E, \CL,\CE)$ can not have a nonempty set $A\in \CE$ and an irreducible path $\bt \in \CL^*(E)$ such that $\CL(AE^{n|\bt|})=\{\bt^n\}$ for all $n \geq 1$. But, a minimal labeled space can have a nonempty set $A\in \CE$ and an irreducible path $\bt \in \CL^*$ such that $\CL(AE^{n|\bt|})=\{\bt^n\}$ for all $n \geq 1$. See the following labeled graph $(E, \CL)$:
\vskip 1pc 
 \hskip 0.5cm
\xy  /r0.3pc/:(-44.2,0)*+{\cdots};(44.2,0)*+{\cdots.};
(-40,0)*+{\bullet}="V-4";
(-30,0)*+{\bullet}="V-3";
(-20,0)*+{\bullet}="V-2";
(-10,0)*+{\bullet}="V-1"; (0,0)*+{\bullet}="V0";
(10,0)*+{\bullet}="V1"; (20,0)*+{\bullet}="V2";
(30,0)*+{\bullet}="V3";
(40,0)*+{\bullet}="V4";
 "V-4";"V-3"**\crv{(-40,0)&(-30,0)};
 ?>*\dir{>}\POS?(.5)*+!D{};
 "V-3";"V-2"**\crv{(-30,0)&(-20,0)};
 ?>*\dir{>}\POS?(.5)*+!D{};
 "V-2";"V-1"**\crv{(-20,0)&(-10,0)};
 ?>*\dir{>}\POS?(.5)*+!D{};
 "V-1";"V0"**\crv{(-10,0)&(0,0)};
 ?>*\dir{>}\POS?(.5)*+!D{};
 "V0";"V1"**\crv{(0,0)&(10,0)};
 ?>*\dir{>}\POS?(.5)*+!D{};
 "V1";"V2"**\crv{(10,0)&(20,0)};
 ?>*\dir{>}\POS?(.5)*+!D{};
 "V2";"V3"**\crv{(20,0)&(30,0)};
 ?>*\dir{>}\POS?(.5)*+!D{};
 "V3";"V4"**\crv{(30,0)&(40,0)};
 ?>*\dir{>}\POS?(.5)*+!D{};
 (-35,1.5)*+{a};(-25,1.5)*+{a};
 (-15,1.5)*+{a};(-5,1.5)*+{a};(5,1.5)*+{a};
 (15,1.5)*+{a};(25,1.5)*+{a};(35,1.5)*+{a};
 \endxy 
\vskip 1pc 
\noindent Then the smallest non-degenerate accommodating  set is $\CE=\{\emptyset, r(a)\}$.  It is easy to see that $(E, \CL,\CE)$ is minimal and $\CL(r(a)E^n)=\{a^n\}$ for all $n \geq 1$. 
\end{remark}

We close this subsection with the following, which will be used to prove Theorem \ref{equiv:simple}.


\begin{prop}\label{cycle at a minimal set} Let $(E,\CL,\CE)$ be a minimal labeled space. If $(\af,A)$ is a cycle with no exits, then $A$ is a minimal set. 
\end{prop}

\begin{proof}  Let $(\af,A)$ be a cycle with no exits. Say $\af=\af_1 \cdots \af_n \in \CL^*(E)$.   We show that if $B \in \CE$ such that $ \emptyset \neq B \subseteq A$, then $B=A$. Since $(E,\CL,\CE)$ is  minimal, we have $\CS(\CH(B))=\CE$ by Theorem \ref{equiv:minimal}. Thus, $A \in \CS(\CH(B))$, and hence, for each $1 \leq i \leq n$,  we have $$r(A,\af_{[1,i]})\subseteq B \cup r(B,\af_1) \cup r(B ,\af_{[1,2]}) \cup \cdots \cup r(B, \af_{[1,n-1]})$$
since $(\af,A)$ and $(\af,B)$ are both cycle without exits.
On the other hand, since $r(A, \af_{[1,i]})\cap r(A, \af_{[1,j]}) =\emptyset$ for $i \neq j$ by \cite[Lemma 6.1]{CaK1} and $r(B, \af_{[1,j]}) \subseteq r(A, \af_{[1,j]})$ for each $1 \leq j \leq n$, it follows that 
$r(A, \af_{[1,i]})\cap r(B, \af_{[1,j]}) =\emptyset$ for $i \neq j$. 
Thus, we have  $r(A, \af_{[1,i]}) \subseteq r(B , \af_{[1,i]})$ for each $1 \leq i \leq n$. It then follows that 
$$A=r(r(A, \af_{[1,i]}) , \af_{[i+1,n]}) \subseteq r( r(B , \af_{[1,i]}) , \af_{[i+1,n]})=B.$$
Hence, $B=A$. 
\end{proof}

\subsection{Disagreeable labeled spaces} 
A notion of disagreeablity of set-finite and  receiver set-finite labeled space $(E, \CL,\CE)$ with  $E$ having no sinks or sources
 was introduced in \cite[Definition 5.1]{BP2} as another analogue notion of Condition (L) of directed graphs. 
If we briefly recall it, 
a labeled path $\af \in \CL([v]_l E^{\geq 1})$  is 
called {\it agreeable} for $[v]_l$ if 
$\af=\bt\af'=\af'\gm$ for some $\af',\bt,\gm \in \CL^*(E)$ 
with $|\bt|=|\gm| \leq l$.
Otherwise $\af$ is called {\it disagreeable}. Note that any  path  $\af$ agreeable for $[v]_l$ 
must be of the form $\af=\bt^k\bt'$ for 
some $\bt\in \CL(E^{\leq l})$,  $k\geq 0$, 
and an initial path $\bt'$ of $\bt$.
 We say that $[v]_l$ is {\it disagreeable} if there is an $N > 0$ 
 such that for all $n > N$ there is an 
 $\af \in \CL(E^{ \geq n})$ that is disagreeable for $[v]_l$. 
A labeled space $(E,\CL, \CE)$ is {\it disagreeable}  
 if  $[v]_l$ is disagreeable for all $v\in E^0$ and $l\geq 1$.

 In \cite{JKP}, they found its equivalent simpler  conditions.  
  
\begin{prop}(\cite[Proposition 3.2]{JKaP}) \label{prop-disagreeable}
For a set-finite and  receiver set-finite labeled space $(E,\CL,\CE)$ with $E$ having no sinks, the following are equivalent:
\begin{enumerate}
\item[(a)] $(E,\CL,\CE)$ is disagreeable.  
\item[(b)] $[v]_l$ is disagreeable for all $v\in E^0$ and $l\geq 1$.
\item[(c)] For each nonempty $A\in \CE$ and a path $\bt\in \CL^*(E)$, 
there is an $n\geq 1$ such that $\CL(AE^{|\bt|n})\neq \{\bt^n\}$.
\end{enumerate}
\end{prop}


   Motivated by Proposition \ref{prop-disagreeable}(c), we define a notion of disagreeability of arbitrary labeled spaces as follows.
   
   \begin{dfn}\label{disagreeable} We say a labeled space $(E,\CL,\CE)$ is {\it disagreeable} if for any nonempty set $A\in \CE$  and a path $\bt\in \CL^*(E)$, 
there is an $n\geq 1$ such that $\CL(AE^{|\bt|n})\neq \{\bt^n\}$.
   \end{dfn}



\begin{dfn}(\cite[Definition 3.2]{JKK}) Let $(E, \CL,  \CE)$ be a  labeled space and  $\af \in \CL^*(E)$ and  
 $\emptyset\neq A \in  \CE$. 
 \begin{enumerate}
 \item  $\af$ is called a {\it loop} at $A$ 
 if  $A \subseteq r(A,\af)$. 
 \item 
A loop  $(\af,A)$ has an {\it exit} 
if one of the following holds:
\begin{enumerate}
\item[(i)] $\{\af_{[1,k]}: 1 \leq k \leq |\af|\} \subsetneq \CL(AE^{\leq |\af|})$.
\item[(ii)] $A\subsetneq r(A,\af)$.
\end{enumerate} 
 \end{enumerate} 

\end{dfn}

By \cite[Proposition 3.7]{JKaP}, one can see that 
 if $(E, \CL,\CE)$ is  disagreeable, then every loop in $(E, \CL,\CE)$ has an exit, and hence, $(E, \CL,\CE)$   satisfies Condition (L). It is shown in \cite[Proposition 3.2]{JP2018} that the other implications are not true, in general. But, if  $(E, \CL,\CE)$ is minimal,  these conditions are equivalent as we see in the following. 

\begin{lem}\label{minimal:prop3}(\cite[Proposition 3.2]{JP2018}) Consider the following three conditions of  a labeled space  $(E,\CL,\CE)$. 
\begin{enumerate}
\item $(E, \CL,\CE)$ is  disagreeable.
\item Every loop in $(E, \CL,\CE)$ has an exit.
\item $(E, \CL,\CE)$   satisfies Condition (L), that is, every cycle has an exit.
\end{enumerate}
Then we have (1)$\implies$(2)$\implies$(3). If, in addition, $(E,\CL,\CE)$ is minimal, (1)-(3) are equivalent. 
\end{lem}

\begin{proof} We only need to show that (3)$\implies$(1) when $(E, \CL,\CE)$ is minimal: Let  $(E, \CL,\CE)$ be minimal. Suppose that $(E, \CL,\CE)$ is not disagreeable. Then there exist a nonempty set $A_0 \in \CE$  and a path $\bt \in \CL^*(E)$ such that for all $n \geq 1$, 
$$\CL(A_0E^{|\bt|n})=\{\bt^n\},$$
where we assume that $\bt$ is irreducible. Then by Lemma \ref{loop}, there is an $N \geq 1$ such that for all $k \geq 1$, 
\begin{align}\label{aa}r(A_0, \bt^{N+k}) \subseteq \cup_{j=1}^N r(A_0, \bt^j).
\end{align}
Take $A:=\cup_{i=1}^N r(A_0,\bt^j)$. Then $A=r(A, \bt)$.  One then can show that $A$ is a minimal set since $(E,\CL,\CE)$ is minimal (see the proof of \cite[Theorem 3.7]{JP2018}). Thus, $(\bt, A)$ is a cycle with no exits. 
Thus, $(E, \CL,\CE)$  does not satisfy Condition (L). 
\end{proof}

 \subsection{Simplicity} We now characterize the simplicity of the $C^*$-algebra associated to an arbitrary  labeled space $(E,\CL,\CE)$. 

 \begin{thm}\label{equiv:simple} Let $(E,\CL,\CE)$ be a labeled space. 
Then the following are equivalent.
\begin{enumerate}
\item $C^*(E,\CL,\CE)$ is simple.
\item  $(E, \CL,\CE)$ is minimal and  satisfies Condition (L).
\item  $(E, \CL,\CE)$ is minimal and  satisfies Condition (K).
\item The following properties hold:
\begin{enumerate}
\item $(E, \CL,\CE)$ is strongly cofinal,
\item $(E, \CL,\CE)$ satisfies Condition (L), and 
\item for any  $ A \in \CE \setminus \{\emptyset\}$ and $B \in \CE$, there is  $C \in \CE_{reg}$ such that  $B \setminus C\in \CH(A)$. 
\end{enumerate}
\item The following properties hold:
\begin{enumerate}
\item $(E, \CL,\CE)$ is strongly cofinal,
\item $(E, \CL,\CE)$ is disagreeable, and 
\item for any  $ A \in \CE \setminus \{\emptyset\}$ and $B \in \CE$, there is  $C \in \CE_{reg}$ such that  $B \setminus C\in \CH(A)$. 
\end{enumerate}
\end{enumerate}
 \end{thm}

  \begin{proof}
  (1)$\implies$(2): If  $C^*(E,\CL,\CE)$ is simple, then 
 the only gauge-invariant ideal of $C^*(E,\CL,\CE)$ is $\{0\}$ and $C^*(E,\CL,\CE)$. Thus, $(E, \CL,\CE)$ is minimal by Theorem \ref{equiv:minimal}.  Suppose that $(E, \CL,\CE)$ does not satisfy Condition (L).  
  Then the labeled space has a cycle $(\af,A)$ with no exits. Since  $(E,\CL,\CE)$ is minimal, $A$ is a minimal set by Proposition \ref{cycle at a minimal set}. Then,  $C^*(E,\CL,\CE)$ has a hereditary subalgebra isomorphic to $M_{|\af|}(C(\T))$ by \cite[Lemma 4.6]{JKP}. It contradicts to  $C^*(E,\CL,\CE)$ is simple.

(2)$\implies$(1):   Let $I$ be a nonzero ideal of $C^*(E,\CL,\CE)$. Since $(E, \CL,\CE)$  satisfies Condition (L), $I$ contains a vertex projection $p_A$ for some $\emptyset \neq A \in \CE$ by the Cuntz-Krieger Uniqueness Theorem \ref{CKUT}. Then $I=C^*(E,\CL,\CE)$ by Theorem \ref{equiv:minimal}.
 
   (2)$\iff$(3):  Follows by  Lemma \ref{LK}.
  
  (2)$\iff$(4): Follows by Theorem  \ref{equiv:minimal}.
  
  (4)$\iff$(5): Follows by Lemma \ref{minimal:prop3}.  
      \end{proof}
 
 As a corollary, we have the following  simplicity results of labeled graph $C^*$-algebras associated to  set-finite labeled spaces with no sinks. 
It  is an improvement on \cite[Theorem 3.7]{JP2018}.             

 \begin{cor}\label{cor:simple} If $E$ has no sinks and $(E,\CL,\CE)$ is set-finite, then the following are equivalent.
\begin{enumerate}
\item $C^*(E,\CL,\CE)$ is simple.
\item   $(E, \CL,\CE)$ is minimal and  satisfies Condition (L).
\item  $(E, \CL,\CE)$ is minimal and  satisfies Condition (K).
\item[(4)] $(E, \CL,\CE)$ is strongly cofinal and  disagreeable.
\item[(5)]  $(E, \CL,\CE)$  is strongly cofinal and   satisfies Condition (L).
\item[(6)]  $(E, \CL,\CE)$  is strongly cofinal and   satisfies Condition (K).
\end{enumerate}
 \end{cor}
 
 \begin{proof}It follows by Corollary \ref{cor:minimal} and Theorem \ref{equiv:simple}.
 \end{proof}
 
 \vskip 2pc
\subsection*{Acknowledgements}  The author wishes to many thank Toke Meier Carlsen for                                                             pointing out her mistake and having a helpful conversation.

\end{document}